\documentclass[11pt,a4]{amsart}
\usepackage[latin1]{inputenc}
\usepackage{a4wide}
\usepackage{amsfonts}
\usepackage{amsmath}
\usepackage{amssymb}
\usepackage{amsthm}
\usepackage{color}
\usepackage{umlaute}
\usepackage{graphicx}
\usepackage{draftcopy}
\usepackage{inputenc}
\usepackage[normal,tight,center]{subfigure}
\usepackage{hyperref}
\hypersetup{
     colorlinks = true,
    citecolor = black,
     linkcolor = black,
     urlcolor=black
}

\newtheorem{theorem}{Theorem}[section]
\newtheorem{remark}[theorem]{Remark}
\newtheorem{lemma}[theorem]{Lemma}
\newtheorem{proposition}[theorem]{Proposition}

\numberwithin{equation}{section}
\numberwithin{equation}{section}
\numberwithin{equation}{section}

\newcommand{\Pbb}[1]{\Pb\lb #1\rb}
\newcommand{\Ebb}[1]{\Eb\lbb #1\rbb}

\newcommand{\LL}{L\'{e}vy }

\newcommand{\ttinf}[1]{_{#1\to \infty}}

\newcommand{\Cb}{\mathbb{C}}
\newcommand{\Eb}{\mathbb{E}}
\newcommand{\Nb}{\mathbb{N}}
\newcommand{\Rb}{\mathbb{R}}
\newcommand{\Pb}{\mathbb{P}}
\newcommand{\Qb}{\mathbb{Q}}

\newcommand{\Cc}{\mathcal{C}}

\newcommand{\Mcc}{\mathcal{M}}

\newcommand{\Tc}{\mathcal{T}}

\newcommand{\ind}[1]{1_{\{#1\}}}

\newcommand{\IntZeroInf}{\int_{0}^{\infty}}
\newcommand{\IntZeroOne}{\int_{0}^{1}}

\newcommand{\lb}{\left (}
\newcommand{\rb}{\right )}
\newcommand{\lbb}{\left [}
\newcommand{\rbb}{\right ]}
\newcommand{\labs}{\left |}
\newcommand{\rabs}{\right |}
\newcommand{\lbcurly}{\left\{}
\newcommand{\rbcurly}{\right\}}
\newcommand{\lbcurlyrbcurly}[1]{\lbcurly#1\rbcurly}
\newcommand{\lbrb}[1]{\lb #1 \rb}
\newcommand{\lbbrbb}[1]{\lbb#1\rbb}
\newcommand{\labsrabs}[1]{\labs#1\rabs}

\newcommand{\suminfinf}{\sum_{j=-\infty}^{\infty}}
\newcommand{\sumoneinf}{\sum_{j=1}^{\infty}}
\newcommand{\eigenfunspecial}[1]{\sin\lbrb{\frac{\pi j}{c} #1}}
\newcommand{\eigenfunction}[1]{\sin\lbrb{#1}}
\newcommand{\eigenvaluespecial}{\frac{\pi^2 j^2}{2c^2}}
\newcommand{\eigenvaluespecialRoot}{\frac{\pi j}{c}}
\newcommand{\expeigenvaluespecial}{e^{-\eigenvaluespecial t}}
\newcommand{\expeigenvalue}[1]{e^{-\frac{\pi^2 j^2}{2}#1}}
\newcommand{\minusone}[1]{\lbrb{-1}^{#1}}
\newcommand{\pij}{\pi^2 j^2}
\newcommand{\Tcc}[2]{1_{\{\Tc_{(0,#1)}>#2\}}}

\newcommand{\IntComplexLine}[2]{\int_{#2-i\infty}^{#2+i\infty}#1}
\newcommand{\twopi}{\frac{1}{2\pi i}}

\begin{document}

\title[Survival distributions of Brownian trajectories among obstacles]{Conditional survival distributions of Brownian trajectories in a one dimensional Poissonian environment in the critical case}

\author{Martin Kolb}
\address{Universit\"at Paderborn, Institut f\"ur Mathematik, Warburger Str. 100, 33098 Paderborn, Germany}
\email{kolb@math.uni-paderborn.de}
\author{{Mladen} {Savov}}
\address{Department of Mathematics and Statistics,
  University of Reading, Whiteknights, Reading RG6 6AX, UK}
\email{m.savov@reading.ac.uk}

\maketitle

\begin{abstract}
In this work we consider a one-dimensional Brownian motion with constant drift moving among a Poissonian cloud of obstacles. Our main result proves convergence of the law of processes conditional on survival up to time $t$ as $t$ converges to infinity in the critical case where the drift coincides with the intensity of the Poisson process. This complements a previous result of T. Povel, who considered the same question in the case where the drift is strictly smaller than the intensity. We also show that the end point of the process conditioned on survival up to time $t$ rescaled by $\sqrt{t}$ converges in distribution to a non-trivial random variable,  as $t$ tends to infinity, which is in fact invariant with respect to the drift $h>0$. We thus prove that it is sub-ballistic and estimate the speed of escape. The latter is in a sharp contrast with discrete models of dimension larger or equal to $2$ when the behaviour at criticality is ballistic, see \cite{IV12}, and even to many one dimensional models which exhibit ballistic behaviour at criticality, see \cite{KM12}.
\end{abstract}


\section{Introduction}
The investigation of stochastic processes in a random environment has along history and is still an active area of research. A very thoroughly studied model is the one of a diffusive particle in a Poissonian environment of obstacles. For a detailed description of the general framework and the mathematical details we refer to the very readable account of this presented in \cite{Sz99}. Our starting point is the following model composed of a one-dimensional Brownian particle $(X_t)_{t\geq 0}$, starting from $0$, with a constant drift $h \neq 0$ and law $W^h$, which moves in an environment given by an independent Poisson process in $\mathbb{R}$ with intensity $\nu$ whose law is denoted by $\mathcal{P}$. Further we denote by $C_t=\sup_{s\leq t}X_s-\inf_{s\leq t}X_s=M_t-m_t$ the range of the process. Let $W^h_t$ be the restriction of the Wiener measure $W^h$ to $C\lbrb{0,t}$. The Brownian particle starts from zero and gets killed upon hitting a point of the Poisson process, i.e. the killing time is denoted by $T$. In this work we will focus on the expected survival time 
$$
W^h \otimes \mathcal{P}\lbcurlyrbcurly{T>t} = \mathbb{E}^h\bigl[e^{-\nu C_t}\bigr]
$$
and in particular on the behaviour of the conditioned law
\begin{displaymath}
\mathbb{Q}_t:=\frac{e^{-\nu C_t}}{\mathbb{E}^h\bigl[e^{-\nu C_t}\bigr]}dW^h_t.
\end{displaymath}
Our special emphasis is on the case where $|h|=\nu$ which due to symmetry can be reduced to $h=\nu$. Before we state our results we recall what is known if $h\neq \nu$ and indicate why this model is of interest. 

Motivated by previous heuristic arguments by physicists and simulation studies (see \cite{MPW84} and \cite{GP82}) it was shown in \cite{ES87} that even for the higher dimensional analogue this model exhibits a phase transition in the sense that 
\begin{equation}\label{e:EL}
\lim_{t \rightarrow \infty}\frac{1}{t}\log\bigl( e^{\frac{h^2}{2}t}\mathbb{E}^h\bigl[e^{-\nu C_t}\bigr]\bigr) = \begin{cases}
\frac{1}{2}(|h|-\nu)^2 &\text{\,if $|h|>\nu$}\\
0 &\text{if $|h| \leq \nu$}.
\end{cases}
\end{equation}
Thus there is a critical parameter regime given by $|h|=\nu$. In \cite{Sz90} it was later demonstrated that in all dimensions
\begin{equation}\label{e:Sz}
\lim_{t \rightarrow \infty}\frac{1}{t^{\frac{1}{3}}}\log\bigl( e^{\frac{h^2}{2}t}\mathbb{E}^h\bigl[e^{-\nu C_t}\bigr]\bigr) = 
\frac{1}{2}(|h|-\nu)^2) \text{ if $|h|<\nu$},
\end{equation}
which gives a much more precise version for the subcritical case $|h|<\nu$ than \eqref{e:EL}. 
The one-dimensional situation was further investigated in more detail by T. Povel in \cite{P95}, where he in particular proved the following result
\begin{theorem}[Theorem A ii) in \cite{P95}]\label{t:P95}
If $|h|<\nu$ then the limiting distribution of the process $X_{\cdot}$ under the measure $\Qb_t$ converges as $t\rightarrow \infty$ to a mixture of Bessel-3-processes under which $X$ starts in $0$ and never hits  a random level $\tilde a$ and the density of the mixture is given by $h^2\tilde ae^{-|h|\tilde a},\,\tilde{a}>0$. 
\end{theorem}
As is explained nicely in \cite{P95} this theorem describes the macroscopic behaviour of the model and as a matter of fact the limit result is different in the case $h=0$. Thus the presence of the drift has an influence on the macroscopic limit. 

Analogues results for the case of a random walk with drift instead of a Brownian motion with drift have been established in \cite{Ve98} and in the case of a Brownian motion with drift moving among soft Poissonian obstacles in \cite{Se03}. In those works a similar exclusion of the critical case is supposed. 

As mentioned in this document we study the critical value model in Povel, i.e. for a Brownian motion with drift $h$ we investigate the convergence of the Brownian motion under the measure
\begin{equation}\label{eq:measure}
\Qb^{(h)}_t=\frac{e^{-h C_t}}{\Eb^h\lbrb{e^{-h C_t}}}W^h_t,
\end{equation}
i.e. we focus on the case $h=\nu$. 

Even though this type of problem has been intensively investigated we have not been able to locate results covering this case in the present literature and it is the aim of the present work to fill in this gap. It will turn out that the macroscopic behaviour is the same as the one in the case $|h|<\nu$ but the details of the proof tend to be much more demanding. Our starting point will be the same as the one of Povel \cite{P95} but we are forced to work a long a different route as already one of his first steps breaks down in the case $h=\nu$. In fact the first task consists in controlling the behavior of the normalization constant $\mathbb{E}^h\bigl[e^{-\nu C_t}\bigr]$ as $t \rightarrow \infty$. In order to establish this Povel \cite{P95} relies on an application of the classical Laplace method, which is not applicable in our setting namely the case $h=\nu$. 
\begin{remark}
In contrast to \cite{P95} our method of analysing the asymptotic behaviour of $\mathbb{Q}^{(h)}_t$ will essentially rely on some facts from the theory of Mellin transforms and on some ideas around the Poisson summation formula. 
\end{remark}
Let us also emphasize that the range of diffusion processes has been of considerable interest in the probability literature (see e.g. \cite{I} and \cite{TV06}) but our main result does not seem to follow easily from these studies. 
\subsection{Comparison with existing literature on Penalizations}
In this subsection we compare our result and methods with several existing results in the area of Penalizations of Diffusions. 
\subsubsection{Differences to existing related results}
The area of penalizations of diffusion processes is of course very classical, but nonetheless there has been quite some further progress mainly due to the efforts of Roynette, Yor, Vallois and  co-authors. The most classical penalization are those of Feynman-Kac-type, i.e. one considers the family of measures
\begin{displaymath}
\mathcal{W}_t:=\frac{1}{C_t}e^{-\int_0^tV(X_s)\,ds}\,\mathbb{W}_t,
\end{displaymath} 
where $C_t=\mathbb{E}\bigl[e^{-\int_0^tV(X_s)\,ds} \bigr]$ and $V$ is a non-negative function. In the case $\int_{\mathbb{R}}|x|V(x)\,dx<\infty$ -- in physics one might want to call those $V$ short range potentials -- the behaviour of the measures $\mathcal{W}_t$ is investigated in detail. These conditioning are more Markovian than our conditioning and due to this a relation to a corresponding eigenvalue problem is quite straightforward. Moreover, due to the decay condition $\int_{\mathbb{R}}|x|V(x)\,dx<\infty$ the potential $V$ is a 'small' perturbation of the Brownian process and a great deal of results on the Sturm-Liouville operators of the form $\frac{1}{2}\frac{d^2}{dx^2}+V$ are available in the literature under this condition. In our situation the conditioning is less Markovian and does not seem to be small in any sense.

Our work involves a penalization of the Brownian paths by the function $f\lbrb{X_t,m_t,M_t}=e^{\nu X_t-\nu \lbrb{M_t-m_t}}$. For this reason we would like to discuss several links with penalizations of Brownian paths which have been well studied in the literature, see for example the book \cite{RoyYor09}. Closer connections to our work can be found in \cite{RoyValYor06}. There the authors consider penalizations with functions of the type $f\lbrb{X_t,m_t,M_t}$ for standard Brownian motion amongst others. To compare to our case take $f\lbrb{X_t,C_t,L_t}=f\lbrb{X_t,m_t,M_t,L_t}$, where $f$ belongs to a specific class of functions, which give rise to interesting martingales. Unfortunately, this class of functions does not include our case $e^{\nu X_t-\nu C_t}$ and therefore their approach does not seem to be applicable in our framework. Our penalization is motivation from applications from the area of random processes in random environments and can therefore be considered to be a very natural penalization example. We do not see how the approach of \cite{RoyValYor06} can be adapted to the problem at hand.

Another instance in the literature that resembles our work is the paper \cite{RoyValYor08}. In this work penalizations of the standard Brownian motion with functions $f\lbrb{m_t,M_t}$ are considered. However, the major condition on $f$ being with compact support is violated in our scenario as we have an exponent of the range $C_t=M_t-m_t$. This seems to be a fundamental difference stemming from the fact that once $f$ is assumed of compact support one works with the Brownian motion restricted to a finite interval or in more generality to a finite set. In this case the semigroup of the Brownian motion restricted to the interval is compact and its asymptotic behaviour is determined by its first eigenvalue and eigenfunction, see the asymptotic results in \cite[Case 2 in Prop.3.5]{RoyValYor08}. In our case if $h=\nu=1$, $\Eb_0\lbrb{e^{X_t-C_t}}\sim 1/t$ which clearly has no exponential decay. This is due to the fact that the range that contributes for the asymptotic of this quantity is of the type $C_t\in\lbbrbb{a\sqrt{t},b\sqrt{t}}$ which precludes due to compactness requirement the possibility for reduction of our problem to \cite[Case 2 in Prop.3.5]{RoyValYor08}. The difference in the problems can also be discerned from the applicable mathematical techniques. Instrumental in \cite{RoyValYor08} is the Laplace method which completely breaks down in our case and we need to use all spectral quantities to derive our results. 

Summarizing one can say that the existing literature on penalizations does not contain our theorems, and even contrary we believe that our results and methods complement and add to the existing ideas developed in the area of penalizations of diffusions. Apart from being able to answer an existing gap in the literature on penalizations as well as diffusion processes in random environments we believe that also the methodology is of additional interest. We would like to emphasize that several new difficulties appear in the analysis of the critical case $h=\nu$ -- usually considered to be the most interesting one -- in contrast to the case $\nu \neq h$ as well as in contrast to the papers on penalizations mentioned above. In particular we want to stress, that it does not seem to be possible to extract only a finite part of the spectral expansion as dominating contribution, instead it seems to be necessary to work with the full eigenfunction expansion in most calculations. This can be seen in particular in the derivation of equation \eqref{eq:Asympt}

Let us end this introduction with some remarks concerning the structure of this work. In the subsequent section 2 we summarize our main results concerning the asymptotic behaviour of the survival distribution and the limit of the conditioned measure $\mathbb{Q}^{(h)}_t$. These results are proved in sections 3 and 4, wherein we make use of the asymptotic properties of certain functions appearing naturally during the proof. The investigation of those functions is deferred to sections 5, 6, 7 and 8. 
\section[Main Results]{Main Results}\label{sec:mainresults}
\subsection{Notation and conventions}\label{subsec:notation}
Throughout the paper we use $f\sim g$ to denote that $\lim f/g=1$ and $f\asymp g$ to imply the existence of two positive constants $C_1<C_2$ such that $C_1 f\leq g\leq C_2 f$.

Throughout the paper we consider a one-dimensional Brownian motion $X$ with drift $h\in\Rb$. We write $W^h$ respectively $W^h_t$ for the Wiener measure on $C\lbrb{0,\infty}$ respectively the restriction of the Wiener measure on $C\lbrb{0,t}$. When $h=0$ we drop the superscript. Similarly, we denote by $\Eb^h_x\lbbrbb{.}$ the expectation of the Brownian motion with drift $h\neq 0$ started from $x$. When $h=0$ we omit the superscript and write instead $\Eb_x\lbbrbb{.}$.

We use the $m_t,\, M_t,\, 0\leq t\leq \infty$ to denote the running minimum, running maximum of $X$, i.e. $m_t=\inf_{s\leq t}X_s$ and $M_t=\sup_{s\leq t}X_s$. We use $C_t$ for the running range of the process, i.e. $C_t=M_t-m_t$.

Due to symmetry \textit{throughout} the paper we assume that $h=\nu>0$.
\subsection{Asymptotic expansion for the Laplace exponent}\label{subsec:AsymptExpansion}
As mentioned before the first crucial quantity to be understood is $\Eb^{(h)}_0\lbbrbb{e^{-\nu C_t}}$ when $\nu=h$ since it is the normalizing constant in the conditioned measure \eqref{eq:measure}. When the drift $\nu\neq h$, \cite[p.223, (4) and (5)]{P95} discusses the precise rate of asymptotic. Furthermore, when $\nu=h$ it follows from \eqref{e:EL} that $\Eb^{(h)}_0\lbbrbb{e^{\frac{h^2 t}{2}}e^{-h C_t}}$ does not grow exponentially and in our one-dimensional situation it is also not difficult to see that $\Eb^{(h)}_0\lbbrbb{e^{\frac{h^2 t}{2}}e^{-h C_t}}\to 0$. Obviously, for our purpose a much stronger control on the rate of decay is necessary and therefore as a first step we provide in Lemma \ref{lem:Asympt}, which is proved in section \ref{subsec:Lem2.1}, a complete asymptotic expansion for the behaviour of $\Eb^{(h)}_0\lbbrbb{e^{-h C_t}}$ as $t\to\infty$. 
\begin{lemma}\label{lem:Asympt}
Let $X$ be a one-dimensional Brownian motion with drift $h>0$. We have the following asymptotic expansion: namely for any $n\in\Nb^+$, as $t\to\infty$,
\begin{equation}\label{eq:Asympt}
\Eb^{h}_0\lbbrbb{e^{-h C_t}}=e^{-h^2\frac{t}{2}}\Eb_0\lbbrbb{e^{X_{th^2}-C_{th^2}}}=e^{-h^2\frac{t}{2}}\lbrb{\frac{1}{th^2}+\sum_{l=1}^{n}\minusone{l}\frac{2^{1}\lbrb{l+1}!}{(th^2)^{l+1}}+o\lbrb{\frac{1}{t^{n+1}}}}.
\end{equation}
\end{lemma}
\begin{remark}
It is useful to compare the assertion of this Lemma with the results \eqref{e:EL} and \eqref{e:Sz}. In the one-dimensional situation the critical case the asymptotic behaviour of the expected survival distribution  differs significantly from the subcritical and the supercritical cases in the fact that the decay has only additionally a polynomial decay factor. Therefore, from this point of view it is not clear, whether the behaviour of $X_t$ under $\mathbb{Q}^{(h)}_t$ in the critical case is similar to the subcritical and the supercritical, respectively, or different from both regimes.  
\end{remark}
\subsection{End point limiting behaviour}
The next result shows that the minimum  $m_t$ under the limiting measure is a non-degenerate random variable and thus in the limit the process is pushed away from $-\infty$.
\begin{theorem}\label{lem:minimumLaw}
For any $A>0$, we have that 
\begin{equation}\label{eq:minimumLaw}
\lim\ttinf{t}\Qb^{(h)}_t\lbrb{-m_t\leq A}=h\int_{0}^{A}e^{-ha}da
\end{equation}
and therefore $\lim\ttinf{t}m_t\stackrel{d}=m_{\infty}$ with  $m_{\infty}\sim Exp(h)$. 
\end{theorem}
In the following theorem we study the joint law of the maximum $M_t=\sup_{s \leq t}X_s$ up to time $t$ and the $X_t$ as $t \rightarrow \infty$. 
\begin{theorem}\label{lem:maximumLaw}
We have that under $\Qb^{(h)}_t$, 
\begin{equation}\label{eq:weakBivariate}
\lim\ttinf{t}\lbrb{\frac{M_t}{\sqrt{t}},\frac{X_t}{\sqrt{t}}}\stackrel{d}{=}\lbrb{M_\infty,M_\infty}.
\end{equation}
where $M_\infty$ has a distribution function which does not depend on $h>0$ and is given by the expression
\begin{equation}\label{eq:T}
T(x)=-G\lbrb{\frac{1}{x^2},\frac{1}{2}}=2\sumoneinf \minusone{j+1}\expeigenvalue{\frac{1}{x^2}},\,x>0,
\end{equation}
where the function $G(\cdot,\cdot)$ is defined in \eqref{eq:F}.
\end{theorem}
\begin{remark}
This result shows that under $\Qb^{(h)}_t$ the process is sub-ballistic and estimates its escape rate, i.e. $\sqrt{t}$. This is in contrast with higher dimensional discrete models of the same type where at criticality the process is ballistic, see \cite{IV12}, and the fact that in one dimension all models but this one are ballistic too, see \cite{KM12}.
\end{remark}
\begin{remark}
Thus in the case $h=\nu$ the position $X_t$ at time $t$ and its maximum up to time $t$ properly rescaled exhibit the same behaviour and even converge to a fully dependent pair of random variables. Moreover, the limiting distribution does not depend on $h$. This eventually follows from the scaling property of the Brownian motion, see section \ref{subsec:Thm2.6}. The same will be valid for the minimum process $m_t$ and $X_t$ provided $h<0$.
\end{remark}
\begin{remark}
It is interesting to note that variants of \eqref{eq:T} appear throughout the review paper \cite{BiPitYor01}. Thus, our random variable $M_\infty$ is a transformation of various quantities such as the maximum of a Brownian bridge, etc., but since we have no further probabilistic explanation as to why these relationships hold we do not discuss the matter further.
\end{remark}
\subsection{Limiting process}
Next we consider the convergence of the process $X$ under the measures $\Qb^{(h)}_t$. Thus, we will specify how the beginning of the process $X$ is affected in the limit by the conditional measures \eqref{eq:measure}.
We have the following result.
\begin{theorem}\label{theorem:BesselMixture}
Under $\Qb^{(h)}_t$ the process $X$ converges to the process $Y$ which is a mixture of (shifted) three dimensional Bessel processes. In more detail, $Y$ is a Brownian motion started from zero and not allowed to hit independent random level $-\tilde a$ whose density is given by $h^2ae^{-ha}da,a>0$. 
\end{theorem}
This result is the analogue of Povel's Theorem \ref{t:P95} for the critical case $|h| = \nu$. It tells us that at the critical case for the initial behaviour of $X$ in the limit under the conditional measures \eqref{eq:measure} there is no transition. 

This is in contrast with the transition of the behaviour of the normalizing quantities $\Eb^h\lbbrbb{e^{-\nu C_t}}$ in \eqref{eq:measure}. Let us point out, that our results clearly demonstrat that the large time behaviour of the the process under the conditional measure differs from the behaviour of the $X_t$ under the conditional measure $\mathbb{Q}_t$, significantly. 

\begin{remark}
Let $h=1$. Note that the exponential law of the global infimum under the limiting measure differs from the law of the barrier $\tilde a$ in Theorem \ref{theorem:BesselMixture}. 
A three dimensional Bessel that is started from $x>0$ then its global infimum is distributed as a uniform random variable on $[0,x]$, see \cite[(8.3.5), p.85]{Don07} wherein $h(x)=x$. Since we start from an independent random level we have that the density of the global minimum is given by $\int_{x}^{\infty}\frac{1}{y}ye^{-y}dydx=e^{-x}dx$ where $1/y$ is the density of the uniform distribution on $[0,y]$ and $ye^{-y}$ is the density of the random level. This shows that the two results are consistent.
\end{remark}
\section{Proofs}\label{sec:proofs}
\subsection{Useful analytical and spectral computations}
We start the proofs by deriving useful formulae and introducing suitable notation. First using Girsanov's theorem and then following \cite[p.226]{P95} our first claim expresses $\Eb^{(h)}_0\lbbrbb{e^{-h C_t}}$ in terms of the double exit times for the Brownian motion with zero drift.
\begin{lemma}\label{lem:keyRangeForm}
Let $X$ be a one-dimensional Brownian motion with drift $h>0$. We have that, for any $0<y\leq \infty$,
\begin{equation}\label{eq:Laplace1}
\Eb^{(h)}_0\lbbrbb{e^{-h C_t}\ind{X_t<y}}=e^{-\frac{h^2}{2}t}\IntZeroInf e^{-c} \int_{0}^{c}e^{-a}\Eb_a\lbrb{e^{X_{th^2}}1_{\{\Tc_{\lbrb{0,c}}>th^2\}}\ind{X_{th^2}<hy}}dadc
\end{equation}
\end{lemma}
\begin{proof}
Using the Girsanov's theorem and then the scaling property of the Brownian motion we rewrite \eqref{eq:Laplace1} as follows
\begin{equation}\label{eq:Girsanov}
\Eb^{(h)}_0\lbbrbb{e^{-h C_t}\ind{X_t<y}}= e^{-\frac{h^2}{2}t}\Eb\lbbrbb{e^{hX_t-hC_t}\ind{X_t<y}}=e^{-\frac{h^2}{2}t}\Eb_0\lbbrbb{e^{X_{th^2}-C_{th^2}}\ind{X_{th^2}<hy}}.
\end{equation}
As in \cite[p.226]{P95} we re-express the quantity
\begin{equation}\label{eq:keyRangeForm}
e^{-C_t}=\int_{M_t}^{\infty}\int_{m_t}^{\infty}e^{-a-b}dadb=\IntZeroInf\IntZeroInf e^{-a-b}1_{\{\Tc_{(-a,b)}>t\}}dadb,
\end{equation}
where $\Tc_{\lbrb{-a,b}}=\inf\{s\geq 0:B_s\notin(-a,b)\}$. Using this and changing variables $c=a+b,a=a$ we get
\begin{equation*}
\Eb^{(h)}_0\lbbrbb{e^{-h C_t}\ind{X_t<y}}=e^{-\frac{h^2}{2}t}\IntZeroInf e^{-c} \int_{0}^{c}e^{-a}\Eb_a\lbrb{e^{X_{th^2}}1_{\{\Tc_{\lbrb{0,c}}>th^2\}}\ind{X_{th^2}<hy}}dadc.
\end{equation*}
\end{proof}

From \eqref{eq:Laplace1} of  Lemma \ref{lem:keyRangeForm} it is obvious that it suffices to work with the case $h=1$. Before proceeding further we evaluate the quantities involved in Lemma \ref{lem:keyRangeForm} using some tools from spectral theory. In the sequel  we denote by $a\wedge b=\min\{a,b\}$ and $a\vee b=\max\{a,b\}$.
\begin{lemma}
Let $X$ be a one-dimensional Brownian motion with drift $h>0$. Recalling that $\Eb^{(0)}_a\lbbrbb{\cdot}=\Eb_a\lbbrbb{\cdot}$, we have, for any $0<y\leq \infty$,
\begin{align}\label{eq:semigroupExpression1}
\nonumber&e^{-c}\Eb_a\lbbrbb{e^{X_t}1_{\{\Tc_{\lbrb{0,c}}>t\}}\ind{X_t<y}}=\\
\nonumber-&2\sumoneinf \expeigenvaluespecial \eigenfunspecial{a}\frac{\pi j}{\pi^2j^2+c^2}\lbrb{e^{(y\wedge c)-c}\cos\lbrb{\frac{\pi j}{c}(y\wedge c)}-e^{-c}}+\\
&2\sumoneinf \expeigenvaluespecial \eigenfunspecial{a}\frac{ c}{\pi^2j^2+c^2}e^{(y\wedge c)-c}\eigenfunspecial{(y\wedge c)},
\end{align}
In more detail when $y=\infty$ we have that
\begin{align}\label{eq:semigroupExpression}
&e^{-c}\Eb_a\lbbrbb{e^{X_t}1_{\{\Tc_{\lbrb{0,c}}>t\}}}=2\sumoneinf \minusone{j+1}\frac{\pi j}{\pi^2j^2+c^2}\expeigenvaluespecial \eigenfunspecial{a}\lbrb{1-\minusone{j}e^{-c}},
\end{align}
and therefore
\begin{align}\label{eq:LaplaceExpression}
&\int_{0}^{c}e^{-a}e^{-c}\Eb_a\lbbrbb{e^{X_t}1_{\{\Tc_{\lbrb{0,c}}>t\}}}da=2c\sumoneinf \minusone{j+1}\frac{\pi^2 j^2}{\lbrb{\pi^2 j^2+c^2}^2}\expeigenvaluespecial\lbrb{1-\minusone{j}e^{-c}}^2,
\end{align}
which implies that
\begin{equation}\label{eq:Laplace2}
\Eb^{(1)}_0\lbbrbb{e^{-C_t}}=e^{-\frac{1}{2}t}\IntZeroInf 2c\sumoneinf \minusone{j+1}\frac{\pi^2 j^2}{\lbrb{\pi^2 j^2+c^2}^2}\expeigenvaluespecial\lbrb{1-\minusone{j}e^{-c}}^2dc.
\end{equation}
\end{lemma}
\begin{proof}
 The semigroup of Brownian motion with zero drift killed at the double exit time $\Tc_{\lbrb{0,c}}=\inf\{s\geq 0:B_s\notin(0,c)\}$ is a compact selfadjoint semigroup and the transition density has the following eigenfunction expansion
\begin{equation}\label{eq:spectralExpansion}
p^{(0,c)}_t(x,y)=p^c_t(x,y)= \frac{2}{c}\sumoneinf \expeigenvaluespecial  \eigenfunction{\eigenvaluespecialRoot x}\eigenfunction{\eigenvaluespecialRoot y},
\end{equation}
for all $x,y\in\lbrb{0,c}$, where $\lambda_j=-\eigenvaluespecial,\,j\geq 1$, are the eigenvalues and $\frac{\sqrt{2}}{\sqrt{c}}\eigenfunspecial{x},\,j\geq 1, x\in(0,c)$, are the normalized eigenfunctions of the operator $\Delta =\frac{1}{2}\frac{d^2}{dx^2}$ with vanishing boundary conditions at $0$ and $c$.
Using \eqref{eq:spectralExpansion} we then easily get upon integration that
\begin{align*}
&e^{-c}\Eb_a\lbrb{e^{X_t}1_{\{\Tc_{\lbrb{0,c}}>t\}}\ind{X_t<y}}=
\frac{2}{c}\sumoneinf \expeigenvaluespecial  \eigenfunction{\eigenvaluespecialRoot a}\int_{0}^{y\wedge c}\eigenfunction{\eigenvaluespecialRoot x}e^xdx.
\end{align*}
Employing the identity 
\begin{equation}\label{eq:integral1}
\int_{0}^{v}\eigenfunspecial{x}e^{x}dx=-\frac{c\pi j}{\pi^2j^2+c^2}\lbrb{e^v\cos\lbrb{\frac{\pi j}{c}v}-1}+\frac{c^2}{\pi^2j^2+c^2}e^{v}\eigenfunspecial{v},
\end{equation}
with $v\leq c$, we derive immediately \eqref{eq:semigroupExpression1} and, plugging $y=\infty$ in \eqref{eq:semigroupExpression1} then \eqref{eq:semigroupExpression} follows.
Using \eqref{eq:semigroupExpression} we can compute that
\begin{align*}
&\int_{0}^{c}e^{-a}e^{-c}\Eb_a\lbrb{e^{X_t}1_{\{\Tc_{\lbrb{0,c}}>t\}}}da=2c\sumoneinf \minusone{j+1}\frac{\pi^2 j^2}{\lbrb{\pi^2 j^2+c^2}^2}\expeigenvaluespecial\lbrb{1-\minusone{j}e^{-c}}^2,
\end{align*}
where we have used 
\begin{equation}\label{eq:integral2}
\int_{0}^{v}\eigenfunspecial{x}e^{-x}dx=-\frac{c\pi j}{\pi^2j^2+c^2}\lbrb{e^{-v}\cos\lbrb{\frac{\pi j}{c}v}-1}-\frac{c^2}{\pi^2j^2+c^2}e^{-v}\eigenfunspecial{v}
\end{equation}
with $v=c$ and the application of the Fubini theorem is immediate since the series \eqref{eq:semigroupExpression} is clearly uniformly convergent for $a\in\lbbrbb{0,c}$, for any fixed $c,t>0$. This is precisely \eqref{eq:LaplaceExpression}.\\
Finally, \eqref{eq:Laplace2} follows by substitution in \eqref{eq:Laplace1} of \eqref{eq:LaplaceExpression} with $h=1$. This completes the proof.
\end{proof}

\subsection{Proof of Lemma \ref{lem:Asympt}}\label{subsec:Lem2.1}
\begin{proof}
To prove Lemma \ref{lem:Asympt} we rewrite \eqref{eq:Laplace2} as follows: first thanks to \eqref{eq:Laplace1} we work with $h=1$ and we change variables $c^2\mapsto w, w\mapsto 1/u, u \mapsto v/t$ to get
\begin{align*}
&\Eb^{(1)}_0\lbbrbb{e^{-C_t}}=e^{-\frac{1}{2}t}\int_{0}^{\infty}2c\sumoneinf \minusone{j+1}\frac{\pi^2 j^2}{\lbrb{\pi^2 j^2+c^2}^2}\expeigenvaluespecial\lbrb{1-\minusone{j}e^{-c}}^2dc=\\
&e^{-\frac{1}{2}t}\int_{0}^{\infty}\sumoneinf \minusone{j+1}\frac{\pi^2 j^2}{\lbrb{\pi^2 j^2u+1}^2} \expeigenvalue{ut}\lbrb{1-\minusone{j}e^{-u^{-1/2}}}^2du=\\
&e^{-\frac{1}{2}t}\frac{1}{t}\int_{0}^{\infty}\sumoneinf \minusone{j+1}\frac{\pi^2 j^2}{\lbrb{\pi^2 j^2\frac{v}{t}+1}^2} \expeigenvalue{v}\lbrb{1-\minusone{j}e^{-t^{1/2}v^{-1/2}}}^2dv.
\end{align*}
We then write
\[I_1(t):=\int_{0}^{\infty}\sumoneinf \minusone{j+1}\frac{\pi^2 j^2}{\lbrb{\pi^2 j^2\frac{v}{t}+1}^2}
  \expeigenvalue{v}dv=\IntZeroInf F(v,t)dv,\]
  with the following definition of the integrand
  \begin{equation}\label{eq:FF}
  F(v,t)=\sumoneinf \minusone{j+1}\frac{\pi^2 j^2}{\lbrb{\pi^2 j^2\frac{v}{t}+1}^2}
    \expeigenvalue{v}
  \end{equation}
  and we put
  \[I_2(t):=\int_{0}^{\infty}\sumoneinf \minusone{j+1}\frac{\pi^2 j^2}{\lbrb{\pi^2 j^2\frac{v}{t}+1}^2}
      \expeigenvalue{v}\lbrb{-2\minusone{j}e^{-t^{1/2}v^{-1/2}}+e^{-2t^{1/2}v^{-1/2}}}dv.\]
 Note that immediately then we get that
 \begin{equation}\label{eq:Laplace3}
 \Eb^{(1)}_0\lbbrbb{e^{-C_t}}=e^{-\frac{1}{2}t}\frac{1}{t}\lbrb{I_1(t)+I_2(t)}.
 \end{equation}  
 \textbf{Study of $I_1(t)$:} From section \ref{sec:FF}, see \eqref{eq:IntFAsympExpansion}, we deduce that, for any $n\in\Nb^+$,
 \begin{equation}\label{eq:I1}
e^{-\frac{1}{2}t}\frac{1}{t} I_1(t)=e^{-\frac{1}{2}t}\frac{1}{t}\IntZeroInf F(v,t)dt= e^{-\frac{1}{2}t}\lbrb{\frac{1}{t}+\sum_{l=1}^{n}\minusone{l}\frac{2^{1}\lbrb{l+1}!}{t^{l+1}}+o\lbrb{\frac{1}{t^{n+1}}}}.
 \end{equation}   
\textbf{Study of $I_2(t)$:}
Let $A(t)=o(t)$, $A(t)\uparrow\infty$ and $t>1$. Estimating from above the function under the integral of $I_2(t)$ we get that
\begin{align*}
&\labsrabs{I_2(t)}=\labsrabs{\int_{0}^{\infty}\sumoneinf \minusone{j+1}\frac{\pi^2 j^2}{\lbrb{\pi^2 j^2\frac{v}{t}+1}^2}
      \expeigenvalue{v}\lbrb{-2\minusone{j}e^{-t^{1/2}v^{-1/2}}+e^{-2t^{1/2}v^{-1/2}}}dv}\leq\\
&4\int_{0}^{A(t)}\sumoneinf \pij \expeigenvalue{v}e^{-\frac{\sqrt{t}}{\sqrt{v}}}dv+4\int_{A(t)}^{\infty}\sumoneinf \pij \expeigenvalue{v}dv\leq\\
&4e^{-\sqrt{\frac{t-1}{A(t)}}}\int_{0}^{A(t)}\sumoneinf \pij e^{-\frac{\pij}{4}v}e^{-\frac{1}{\sqrt{v}}}dv +4e^{-\frac{A(t)}{4}} \int_{A(t)}^{\infty}\sumoneinf \pij e^{-\frac{\pij}{4}v}dv\leq\\
&4\max\left\{e^{-\sqrt{\frac{t-1}{A(t)}}};e^{-\frac{A(t)}{4}}\right\}\lbrb{\int_{0}^{\infty}\sumoneinf \pij e^{-\frac{\pij}{4}v}e^{-\frac{1}{\sqrt{v}}}dv+\int_{1}^{\infty}\sumoneinf \pij e^{-\frac{\pij}{4}v}dv}=\\
&4\max\lbcurlyrbcurly{e^{-\sqrt{\frac{t-1}{A(t)}}};e^{-\frac{A(t)}{4}}}\lbrb{-\frac{1}{2}\int_{0}^{\infty}G'\lbrb{\frac{v}{2},0}e^{-\frac{1}{\sqrt{v}}}dv-\frac{1}{2}\int_{1}^{\infty}G'\lbrb{\frac{v}{2},0}dv},
\end{align*}
where the function $G'(\frac{v}{2},0)$ under the integral is computed from \eqref{eq:F} of section \ref{sec:PS}.
Clearly  then the asymptotic relation \eqref{eq:GsmallAsymp1}, which yields $\labsrabs{G'\lbrb{\frac{v}{2},0}}\sim \frac{4}{\sqrt{\pi}v^{3/2}}$, when $v\to 0$, and an obvious computation
\[\int_{1}^{\infty}\sumoneinf \pij e^{-\frac{\pij}{4}v}dv=4\sum_{j=1}^{\infty}e^{-\pij}<\infty\]
 imply that
\[\int_{0}^{\infty}\labsrabs{G'\lbrb{\frac{v}{2},0}}e^{-\frac{1}{\sqrt{v}}}dv+\int_{1}^{\infty}\labsrabs{G'\lbrb{\frac{v}{2},0}}dv<\infty\]
and thus, for any $n\in\Nb^+$,
\begin{equation}\label{eq:lemI_2Ingredient2}
e^{\frac{1}{2}t}I_2(t)=o\lbrb{\frac{1}{t^n}}
\end{equation}
Therefore using \eqref{eq:lemI_2Ingredient2} and \eqref{eq:I1} in \eqref{eq:Laplace3} lead to our claim \eqref{eq:Asympt} where we just recall that when $h\neq 1$ we use relation \eqref{eq:Laplace1} which is basically a rescaling of the time.
 \end{proof}     
\subsection{Proof of Theorem \ref{lem:minimumLaw}}

The proof of Theorem \ref{lem:minimumLaw} and subsequent results hinge on the following key result.
\begin{lemma}\label{lem:Density}
Let $a>0$. We have that
\begin{equation}\label{eq:lemDensity}
\lim\ttinf{t}t\int_{a}^{\infty} e^{-c}\Eb_a\lbbrbb{e^{X_t}1_{\{\Tc_{(0,c)}>t\}}}dc=a.
\end{equation}
Moreover we have uniform convergence, namely
\begin{equation}\label{eq:lemDensity1}
\lim\ttinf{t}\sup_{b\leq a}\labsrabs{t\int_{b}^{\infty} e^{-c}\Eb_b\lbbrbb{e^{X_t}1_{\{\Tc_{(0,c)}>t\}}}dc-b}=0.
\end{equation}
\end{lemma}
\begin{proof}
From \eqref{eq:semigroupExpression} we get that
\begin{align*}
&\int_{a}^{\infty} e^{-c}\Eb_a\lbbrbb{e^{X_t}1_{\{\Tc_{\lbrb{0,c}}>t\}}}dc=\\
&2\int_{a}^{\infty}\sumoneinf \minusone{j+1}\frac{\pi j}{\pi^2j^2+c^2}\expeigenvaluespecial \eigenfunspecial{a}\lbrb{1-\minusone{j}e^{-c}}dc.
\end{align*}
Changing variables $u:=t/c^2$ we get
\begin{align}\label{eq:lemDensityToBeProved}
\nonumber&t\int_{a}^{\infty} e^{-c}\Eb_a\lbbrbb{e^{X_t}1_{\{\Tc_{\lbrb{0,c}}>t\}}}dc=\\
\nonumber&\frac{t}{\sqrt{t}}\int_{0}^{\frac{t}{a^2}} \frac{1}{\sqrt{u}}\sumoneinf \minusone{j+1}\frac{\pi j}{\pi^2j^2\frac{u}{t}+1}\expeigenvalue{u} \eigenfunction{\pi j a \frac{\sqrt{u}}{\sqrt{t}}}\lbrb{1-\minusone{j}e^{-\frac{\sqrt{t}}{\sqrt{u}}}}du=\\
&tJ_1(t,a)+tJ_2(t,a),
\end{align}
where \eqref{eq:J_1General} and \eqref{eq:J_2Recall} with $\nu=\infty$ are the expressions of $J_1(t,a), J_2(t,a)$ given by
\begin{align}\label{eq:J_1}
\nonumber &J_1(t,a):=\frac{1}{\sqrt{t}}\int_{0}^{\frac{t}{a^2}} \frac{1}{\sqrt{u}}\sumoneinf \minusone{j+1}\frac{\pi j}{\pi^2j^2\frac{u}{t}+1}\expeigenvalue{u} \eigenfunction{\pi j a \frac{\sqrt{u}}{\sqrt{t}}}du=\\
&\frac{1}{\sqrt{t}}\int_{0}^{\frac{t}{a^2}}\frac{1}{\sqrt{u}}H\lbrb{a,u,\frac{1}{t},1,\frac{\sqrt{u}}{\sqrt{t}},1}du,
\end{align}
\begin{align}\label{eq:J_2}
\nonumber&J_2(t,a):=\frac{1}{\sqrt{t}}\int_{0}^{\frac{t}{a^2}} \frac{1}{\sqrt{u}}\sumoneinf \frac{\pi j}{\pi^2j^2\frac{u}{t}+1}\expeigenvalue{u} \eigenfunction{\pi j a \frac{\sqrt{u}}{\sqrt{t}}}e^{-\frac{\sqrt{t}}{\sqrt{u}}}du=\\
&\frac{1}{\sqrt{t}}\int_{0}^{\frac{t}{a^2}}\frac{1}{\sqrt{u}}H\lbrb{a,u,\frac{1}{t},1,\frac{\sqrt{u}}{\sqrt{t}},0}du
\end{align}
Note that  for $s\in\{0,1\}$ the function $H$ is defined by
\begin{equation}\label{eq:H}
H\lbrb{a,u,\rho,\gamma,h,s}=(-1)^s\sumoneinf \frac{\pi j}{\pi^2j^2u\rho+1}\expeigenvalue{u\gamma}\sin\lbrb{\pi js+\pi j a h}.
\end{equation}
Relations \eqref{eq:lemDensity} and \eqref{eq:lemDensity1} follow from the representation \eqref{eq:lemDensityToBeProved}, the application Lemma \ref{lem:J_1} with $\nu=\infty$  ( yielding $tJ_1(t,a)\to a$ uniformly on $a-$compact intervals ) and Lemma \ref{lem:J2}  with $\nu=\infty$ ( yielding $tJ_2(t,a)=o(1)$ uniformly on $a\geq 0$ ) 
and the fact that $G\lbrb{0,\frac{1}{2}}=-1$, see \eqref{eq:FPoisson} for more detail.
\end{proof}
Our next lemma improves the result above in a sense that it allows to truncate the integral from $\ln(t)$. This will be useful when we wish to remove the dependence on $a$ at the lower limit of the integrals in \eqref{eq:lemDensity} and \eqref{eq:lemDensity1}.
\begin{lemma}\label{lem:improvedConvergence}
As $t\to\infty$,
\begin{align}\label{eq:improvedConvergence}
\nonumber&\sup_{a\leq \ln(t)}t\labsrabs{\int_{\ln(t)}^{\infty}e^{-c}\Eb_a\lbbrbb{e^{X_t}\Tcc{c}{t}}dc-a}=\\
&\sup_{a\leq \ln(t)}t\labsrabs{\int_{a}^{\infty}e^{-c}\Eb_a\lbbrbb{e^{X_t}\Tcc{c}{t}}dc-a}+o\lbrb{e^{-\frac{\pi t}{4\ln^2(t)}}}
\end{align}
\end{lemma}
\begin{proof}
We observe that from the spectral expansion \eqref{eq:spectralExpansion}
\begin{align*}
\sup_{a\in(0,c)}\Pb_a\lbrb{\Tc_{(0,c)}>t}\leq \frac{\sqrt{2}}{\sqrt{c}}\sumoneinf \expeigenvaluespecial.
\end{align*}
Now we derive an easy estimate by splitting the supremum over $(0,c)$ and giving appropriate estimates for the cases $c\in(0,1)$ and $c\in(1,\ln(t))$
\begin{align*}
&\sup_{0\leq c\leq \ln(t)}\sup_{a\in(0,c)}\Pb_a\lbrb{\Tc_{(0,c)}>t}\leq \sup_{0\leq c\leq \ln(t)}\frac{\sqrt{2}}{\sqrt{c}}\sumoneinf e^{-\frac{\pi j}{2c^2}t}=\sup_{0\leq c\leq \ln(t)}\frac{\sqrt{2}}{\sqrt{c}}\frac{e^{-\frac{\pi}{2c^2}t}}{1-e^{-\frac{\pi}{2c^2}t}}\leq\\
&\sup_{0\leq c\leq \ln(t)}\frac{\sqrt{2}}{\sqrt{c}}\frac{e^{-\frac{\pi }{2c^2}t}}{1-e^{-\frac{\pi }{2\ln^2(t)}t}}\leq C\max\lbcurlyrbcurly{e^{-\frac{\pi}{2}(t-1)}\sup_{0\leq c\leq 1}\frac{1}{\sqrt{c}}e^{-\frac{1}{c^2}},e^{-\frac{\pi t}{2\ln^2(t)}}}=O\lbrb{e^{-\frac{\pi t}{2\ln^2(t)}}}.
\end{align*}
 Therefore using the elementary bound $e^{X_t}\leq e^c$ valid on $t<\Tc_{(0,c)}$  we obtain that
\begin{align}\label{eq:estimatefoUtilde}
&\sup_{a\leq \ln(t)}t\labsrabs{\int_{a}^{\ln(t)}e^{-c}\Eb_a\lbbrbb{e^{X_t}\Tcc{c}{t}}dc}\leq t\ln(t)\sup_{0\leq c\leq \ln(t)}\sup_{b\in(0,c)}\Pb_b\lbrb{\Tc_{(0,c)}>t}=o\lbrb{e^{-\frac{\pi t}{4\ln^2(t)}}}
\end{align}
and the claim follows.
\end{proof}
Now we are ready to start with the proof of Theorem \ref{lem:minimumLaw}.

\begin{proof}[Proof of Theorem \ref{lem:minimumLaw}:]
The result is an easy consequence of Lemma \ref{lem:Density}. Recall that by the definition of $\Qb^{(h)}$ and then \eqref{eq:Girsanov}, for any $A>0$,
\begin{align*}
&\Qb^{(h)}_t\lbrb{-m_t\leq A}=\frac{\Eb^{(h)}\lbbrbb{e^{-hC_t}1_{\{-m_t\leq A\}}}}{\Eb^{(h)}\lbbrbb{e^{-hC_t}}}=\frac{\Eb_0\lbbrbb{e^{X_{th^2}-C_{th^2}}1_{\{-m_{th^2}\leq Ah\}}}}{\Eb_0\lbbrbb{e^{X_{th^2}-C_{th^2}}}}.
\end{align*}
From the representation \eqref{eq:keyRangeForm} together with $1_{\lbrace\Tc_{(-a,b)}>th^2\rbrace}\times 1_{\lbrace -m_{th^2}\leq Ah\rbrace}=1_{\lbrace\Tc_{(-(a\wedge Ah),b)}>th^2\rbrace}$
for the numerator and \eqref{eq:Asympt} for the denominator  we easily get that
\begin{align}\label{eq:minimum1}
\nonumber&\Qb^{(h)}_t\lbrb{-m_t\leq A}=\frac{\IntZeroInf \int_{0}^{\infty}e^{-b-a}\Eb_0\lbbrbb{e^{X_{th^2}}\Tc_{(-(a\wedge Ah),b)}>th^2}dadb}{\Eb_0\lbrb{e^{X_{th^2}-C_{th^2}}}}\sim\\
&h^2t\IntZeroInf \int_{0}^{\infty}e^{-b-a}\Eb_0\lbbrbb{e^{X_{th^2}}\Tc_{(-(a\wedge Ah),b)}>th^2}dadb.
\end{align}
Shifting the starting point from $0\mapsto a\wedge Ah$ for the zero mean Brownian motion under $\Ebb{.}$ we get
\begin{align}\label{eq:minimum2}
\nonumber&\Qb^{(h)}_t\lbrb{-m_t\leq A}\sim h^2t\IntZeroInf \int_{0}^{\infty}e^{-b-a-a\wedge Ah}\Eb_{a\wedge Ah}\lbbrbb{e^{X_{th^2}},\Tc_{(0,b+(a\wedge Ah))}>th^2}dadb=\\
\nonumber&h^2t\IntZeroInf \int_{0}^{Ah}e^{-b-2a}\Eb_a\lbbrbb{e^{X_{th^2}},\Tc_{(0,b+a)}>th^2}dadb+\\
\nonumber&th^2\IntZeroInf e^{-b-2Ah}\Eb_{Ah}\lbbrbb{e^{X_{th^2}},\Tc_{(0,b+Ah)}>th^2}db=\\
&\int_{0}^{Ah}e^{-a}th^2\lbrb{\int_{a}^{\infty}e^{-c}\Eb_a\lbbrbb{e^{X_{th^2}},\Tc_{(0,c)}>th^2}dc}da+e^{-Ah}th^2\int_{Ah}^{\infty}e^{-c}\Eb_{Ah}\lbbrbb{e^{X_{th^2}},\Tc_{(0,c)}>th^2}dc.
\end{align}
Now the uniform convergence in \eqref{eq:lemDensity1} of Lemma \ref{lem:Density} shows that the DCT is applicable to the last two expression yielding that
\[\lim\ttinf{t}\Qb^{(h)}_t\lbrb{-m_t\leq A}=\int_{0}^{Ah}ae^{-a}da+Ahe^{-Ah}=h\int_{0}^{A}e^{-ah}da.\] 
This is valid for any $A>0$ and we note that $he^{-ha}da$ is the probability density of $Exp(h)$. This concludes our claim.
\end{proof}

\subsection{Proof of Theorem \ref{lem:maximumLaw}}\label{subsec:Thm2.6}

\begin{proof}[Proof of Theorem \ref{lem:maximumLaw}:]
Choose $\nu>0$ and we consider as in the proof of Theorem \ref{lem:minimumLaw}
\begin{align*}
&\Qb^{(h)}_t\lbrb{M_t\leq \nu\sqrt{t}}=\frac{\Eb_0\lbbrbb{e^{X_{th^2}-C_{th^2}}\ind{M_{th^2}\leq h\nu\sqrt{t}} }}{\Eb_0\lbbrbb{e^{X_{th^2}-C_{th^2}}}}\sim\\ &h^2t\Eb_0\lbbrbb{e^{X_{th^2}-C_{th^2}}\ind{M_{th^2}\leq h\nu\sqrt{t}}}
\sim \Qb^{(1)}_{th^2}\lbrb{M_{th^2}\leq \nu\sqrt{th^2}} .
\end{align*}
Therefore we note that any possible limit will be invariant with respect to $h$. Hence, assume that $h=1$. An easy computation involving the representation \eqref{eq:keyRangeForm} and $\ind{\Tc_{(-a,b)}>t}\times \ind{M_{t}\leq \nu\sqrt{t}}=\ind{\Tc_{(-a,b\wedge \nu\sqrt{t})}>t}$ and shift of the starting position of $X$ from $0\mapsto a$ yield that
\begin{align}\label{eq:sameformula}
\nonumber&\Eb_0\lbbrbb{e^{X_{t}-C_{t}}\ind{M_{t}\leq\nu\sqrt{t}}}=\IntZeroInf\IntZeroInf e^{-b-a}\Eb_0\lbbrbb{e^{X_{t}},\Tc_{(-a,b\wedge \nu\sqrt{t})}>t}dadb=\\
\nonumber&\int_{0}^{\nu\sqrt{t}}\IntZeroInf e^{-b-2a}\Eb_a\lbbrbb{e^{X_{t}},\Tc_{(0,b+a)}>t}dadb+e^{-\nu\sqrt{t}}\IntZeroInf e^{-2a}\Eb_a\lbbrbb{e^{X_{t}},\Tc_{(0,\nu\sqrt{t}+a)}>t}da=\\
\nonumber&\IntZeroInf e^{-a}\int_{a}^{a+\nu\sqrt{t}} e^{-c}\Eb_a\lbbrbb{e^{X_{t}},\Tc_{(0,c)}>t}dcda+e^{-\nu\sqrt{t}}\IntZeroInf e^{-2a}\Eb_a\lbbrbb{e^{X_{t}},\Tc_{(0,\nu\sqrt{t}+a)}>t}da=\\
&Y_t(\nu)+O_t(\nu).
\end{align}
Let us first study $O_t(\nu)$. Choose $A>0$. We have that
\begin{align}\label{eq:O}
O_t(\nu)=&e^{-\nu\sqrt{t}}\int_0^{A\sqrt{t}} e^{-2a}\Eb_a\lbbrbb{e^{X_{t}},\Tc_{(0,\nu\sqrt{t}+a)}>t}da+\\
\nonumber&e^{-\nu\sqrt{t}}\int_{A\sqrt{t}}^\infty e^{-2a}\Eb_a\lbbrbb{e^{X_{t}},\Tc_{(0,\nu\sqrt{t}+a)}>t}da.
\end{align}
Note that 
\[e^{-\nu\sqrt{t}-a}\Eb_a\lbbrbb{e^{X_{t}},\Tc_{(0,\nu\sqrt{t}+a)}>t}\leq 1\]
and henceforth 
\begin{equation}\label{eq:O-1}
e^{-\nu\sqrt{t}}\int_{A\sqrt{t}}^\infty e^{-2a}\Eb_a\lbbrbb{e^{X_{t}},\Tc_{(0,\nu\sqrt{t}+a)}>t}da\leq e^{-A\sqrt{t}}.
\end{equation}
To study the first term in \eqref{eq:O} we use \eqref{eq:semigroupExpression} with $c=\nu\sqrt{t}+a$, the fact that $a\leq A\sqrt{t}$, $\sin(x)\leq |x|$ and $(a+b)^{-1}\leq a^{-1}, a>0,b>0$ to obtain that
\begin{align}\label{eq:O-2}
\nonumber&e^{-\nu\sqrt{t}-a}\Eb_a\lbbrbb{e^{X_{t}},\Tc_{(0,\nu\sqrt{t}+a)}>t}=\\
\nonumber&2\sumoneinf\minusone{j+1}\frac{\pi j}{\pij+\lbrb{a+\nu\sqrt{t}}^2}e^{-\frac{\pij}{2\lbrb{a+\nu\sqrt{t}}^2}t}\sin\lbrb{\frac{\pi j a}{a+\nu\sqrt{t}}}\lbrb{1-\minusone{j}e^{-a-\nu\sqrt{t}}}\leq\\
&\frac{4a}{\lbrb{a+\nu\sqrt{t}}^3}\sumoneinf \pij e^{-\frac{\pij}{2\lbrb{\nu+A}^2}}\leq C\frac{a}{\nu^3t^{\frac{3}{2}}}.
\end{align}
Therefore from \eqref{eq:O-1} and \eqref{eq:O-2} we deduce in \eqref{eq:O} that
$tO_t(\nu)=o\lbrb{1}$ and hence 
\[\Qb^{(1)}_t\lbrb{M_t\leq \nu\sqrt{t}}\sim t Y_t(\nu)=t\IntZeroInf e^{-a}\int_{a}^{a+\nu\sqrt{t}} e^{-c}\Eb_a\lbbrbb{e^{X_{t}},\Tc_{(0,c)}>t}dcda.\]
However, we proceed with the same steps leading to \eqref{eq:lemDensityToBeProved} in the proof of Lemma \ref{lem:Density} to get with obvious modification coming from integrating between $(a,a+\nu\sqrt{t})$ in the inner integral
\begin{align*}
&Y_t(\nu)=\int_{0}^{\infty}e^{-a}\lbrb{J_1(t,a,\nu)+J_2(t,a,\nu)}da,
\end{align*}
where 
\[J_1(t,a,\nu):=\frac{1}{\sqrt{t}}\int_{\frac{t}{\lbrb{a+\nu\sqrt{t}}^2}}^{\frac{t}{a^2}} \frac{1}{\sqrt{u}}\sumoneinf \minusone{j+1}\frac{\pi j}{\pi^2j^2\frac{u}{t}+1}\expeigenvalue{u} \eigenfunction{\pi j a \frac{\sqrt{u}}{\sqrt{t}}}du;\]
\[J_2(t,a,\nu):=\frac{1}{\sqrt{t}}\int_{\frac{t}{\lbrb{a+\nu\sqrt{t}}^2}}^{\frac{t}{a^2}} \frac{1}{\sqrt{u}}\sumoneinf \frac{\pi j}{\pi^2j^2\frac{u}{t}+1}\expeigenvalue{u} \eigenfunction{\pi j a \frac{\sqrt{u}}{\sqrt{t}}}e^{-\frac{\sqrt{t}}{\sqrt{u}}}du\]
are defined and studied in sections \ref{sec:J_1} and \ref{sec:J_2}. From \eqref{eq:J2Bound} and DCT we obtain that \[\int_{0}^{\infty}e^{-a}\labsrabs{tJ_2(t,a,\nu)}da=o\lbrb{1}.\] 
Choose $a(t)=t^{\frac{1}{6}}$ then
\begin{align}\label{eq:Y1}
\nonumber&\int_{0}^{a(t)}e^{-a}J_1(t,a,\nu)da=\int_{0}^{a(t)}e^{-a}a\int_{\frac{1}{\nu^2}}^{\infty} \frac{2}{u}\frac{\partial \Gamma}{\partial \gamma}\lbrb{0,u,0,1,0,1}duda+r(t)=\\
&-G\lbrb{\frac{1}{\nu^2},\frac{1}{2}}\int_{0}^{a(t)}e^{-a}ada+r(t),
\end{align}
where the first identity follows  from Proposition \ref{prop:J_1Improvement} where the second comes from \eqref{eq:derivativeGamma2}, \eqref{eq:F} and $G(\infty,\frac{1}{2})=0$ which in turn give
\[\frac{2}{u}\frac{\partial \Gamma}{\partial \gamma}\lbrb{0,u,0,1,0,1}=\sum_{j=1}^{\infty}\minusone{j+1}\pij e^{-\frac{\pij}{2}u}=G'\lbrb{u,\frac{1}{2}}\]
From \eqref{eq:J1errorImprovement1}
\[r(t)\leq C\lbrb{t^{-\frac{1}{2}}+t^{-\frac{1}{6}}}\int_{0}^{a(t)}e^{-a}da\lesssim t^{-\frac{1}{6}}.\]
Henceforth,  we get that
\[\lim\ttinf{t}\int_{0}^{a(t)}e^{-a}J_1(t,a,\nu)da=-G\lbrb{\frac{1}{\nu^2},\frac{1}{2}}\IntZeroInf ae^{-a}da =-G\lbrb{\frac{1}{\nu^2},\frac{1}{2}}.\]
This proves \eqref{eq:T} since from Proposition \ref{prop:J1UniversalBound} we get that
\begin{align*}
&t\int_{a(t)}^{\infty}e^{-a}\labsrabs{J_1(t,a,\nu)}da\leq \int_{a(t)}^{\infty}e^{-a}\frac{2}{a}\labsrabs{G'\lbrb{\frac{1}{\lbrb{\frac{a}{\sqrt{t}}+\nu}^2},0}}da\leq\\
&2te^{-a(t)}\int_{0}^{\infty}e^{-w}G'\lbrb{\frac{1}{\lbrb{\frac{w+a(t)}{\sqrt{t}}+\nu}^2},0}dw\leq\\
& Cte^{-t^{\frac{1}{6}}}\int_{0}^{\infty}e^{-w} \lbrb{1\vee \lbrb{\frac{w+a(t)}{\sqrt{t}}+\nu}^3}dw=o\lbrb{1},
\end{align*}
where we have used that from \eqref{eq:GsmallAsymp1} we get that $|G'\lbrb{v,0}|\asymp \frac{1}{v^{3/2}}$, as $v\to 0$.\\
We then observe that for any $\nu>\vartheta>0$
\begin{align*}
&\Qb^{(h)}_t\lbrb{M_t>\nu\sqrt{t};X_t\leq \vartheta \sqrt{t}}=\frac{\Eb_0\lbbrbb{e^{X_{th^2}-C_{th^2}}\ind{M_{th^2}> \nu\sqrt{th^2};X_{th^2}\leq \vartheta\sqrt{th^2}} }}{\Eb_0\lbbrbb{e^{X_{th^2}-C_{th^2}}}}\sim\\ &th^2\Eb_0\lbbrbb{e^{X_{th^2}-C_{th^2}}\ind{M_{th^2}> \nu\sqrt{th^2};X_{th^2}\leq \vartheta\sqrt{th^2}}}\sim \Qb^{(1)}_{th^2}\lbrb{M_{th^2}>\nu\sqrt{th^2};X_{th^2}\leq \vartheta \sqrt{th^2}}.
\end{align*}
Then without loss of generality put $h=1$. However, using \eqref{eq:keyRangeForm} to express $e^{-C_t}$ and the same computation as in \eqref{eq:sameformula}  we get that
\begin{align*}
&\Eb_0\lbbrbb{e^{X_t-C_t}\ind{M_t> \nu\sqrt{t};X_t\leq \vartheta\sqrt{t}}}= %
\Eb_0\lbbrbb{e^{X_t}\ind{X_t\leq \vartheta\sqrt{t}}\IntZeroInf\IntZeroInf e^{-a-b}\ind{M_{t}>\nu\sqrt{t}}\ind{\Tc_{(-a,b)}>t}dadb}=\\
&\IntZeroInf e^{-a}\int_{a+\nu\sqrt{t}}^{\infty} e^{-c}\Eb_a\lbbrbb{e^{X_t}\ind{X_t\leq \vartheta\sqrt{t}},\Tc_{(0,c)}>t}dcda-\\
&e^{-\nu\sqrt{t}}\IntZeroInf e^{-2a}\Eb_a\lbbrbb{e^{X_t}\ind{X_t\leq \vartheta\sqrt{t}},\Tc_{(0,\nu\sqrt{t}+a)}>t}da=\\
&Y_{t}\lbrb{\nu,\vartheta}+O_t\lbrb{\nu,\vartheta}.
\end{align*}
Exactly as the proof of $tO_t(\nu)=o\lbrb{1}$ we get $tO_t(\nu,\vartheta)=o\lbrb{1}$, namely that the second integral is irrelevant for the asymptotic. However, noting that 
\[\Eb_a\lbbrbb{e^{X_t}\ind{X_t\leq \vartheta\sqrt{t}},\Tc_{(0,c)}>t}\leq e^{\vartheta \sqrt{t}}\Pb_a\lbrb{\Tc_{(0,c)}>t}\leq e^{\vartheta \sqrt{t}} \]
we get that
\begin{align*}
&tY_{t}\lbrb{\nu,\vartheta}\leq t\IntZeroInf e^{-a}\int_{a+\nu\sqrt{t}}^{\infty} e^{-c+\vartheta\sqrt{t}}dcda=\\
&te^{-\lbrb{\nu-\vartheta}\sqrt{t}}\IntZeroInf e^{-2a}da=o\lbrb{1}
\end{align*}
and this shows that, for any pair $\nu>\vartheta>0$, we have that
\[\lim\ttinf{t}\Qb^{(h)}_t\lbrb{M_t>\nu\sqrt{t};X_t\leq \vartheta \sqrt{t}}=0.\]
This proves that $\lim\ttinf{t}\lbrb{\frac{X_t}{\sqrt{t}},\frac{M_t}{\sqrt{t}}}\stackrel{d}{=}\lbrb{M_\infty,M_\infty}$.
\end{proof}

\section{Proof of Theorem \ref{theorem:BesselMixture}}\label{sec:TheoremBesselMixture}
\subsection{Preliminaries and notation}
We recall that a three dimensional Bessel process $Y^a$ started from $a\geq 0$ is a stochastic process with continuous paths. It describes the radial part of a three dimensional Brownian motion started from $a$ and can be identified with a Brownian motion started from $a\geq 0$ conditioned not to cross zero. We denote by $\Pb^{\dagger}_a$ the canonical measure induced by $Y^a$ on the space $\Cb\lbrb{0,\infty}$. We recall that the scaling property of the Bessel process translates as follows: for any bounded functional $F:\Cb\lbrb{0,\infty}\mapsto \Rb$, $h>0,a\geq 0$
\begin{equation}\label{eq:scalingBessel}
\Eb^\dagger_a\lbbrbb{F\lbrb{X_{.h^2}}}=\Eb^\dagger_{\frac{a}{h}}\lbbrbb{F\lbrb{hX_{.}}}.
\end{equation}
Furthermore, if $F:\Cb\lbrb{0,u}\mapsto \Rb$, $a>0,u>0,x>0$ then
\begin{equation}\label{eq:semigroupBessel}
\Eb^\dagger_a\lbbrbb{F\lbrb{X_{.}}\ind{X_u\in dx}}=\frac{x}{a}\Eb_a\lbbrbb{F\lbrb{X_{.}}\ind{X_u\in dx},\Tc_{\lbrb{0,\infty}}>u},
\end{equation}
where $\Tc_{\lbrb{0,\infty}}$ is the first exit from the half-line $\lbrb{0,\infty}$, see \cite[(8.3.2) p.83]{Don07} which applies with $h(x)=x$ in the case of zero drift Brownian motion.
\subsection{Proof of Theorem \ref{theorem:BesselMixture}}\label{subsec:TheoremBesselMixture}
\begin{proof}[Proof of Theorem \ref{theorem:BesselMixture}]
Fix $u>0$ and a bounded, continuous functional $F:=F_u:\Cb\lbrb{0,u}\mapsto \Rb^+$ with $||F||_\infty$ its supremum norm. Choose $B>0$ and let in the sequel $x\in[-B,B]$. Denote by $\Eb^{\Qb_t^{(h)}}$ the expectation under $\Qb_t^{(h)}.$  Choose $A>2B\lbrb{h+h^{-1}}$ and write
\begin{align}\label{eq:decomposition}
\nonumber&\Eb^{\Qb^{(h)}_t}\lbbrbb{F(X_.)\ind{X_u\in dx}}=\\
\nonumber&\Eb^{\Qb^{(h)}_t}\lbbrbb{F(X_.)\ind{X_u\in dx}\ind{m_t\geq-A}}+\Eb^{\Qb^{(h)}_t}\lbbrbb{F(X_.)\ind{X_u\in dx}\ind{m_t\leq-A}}=\\
&U_{t,h}(dx,A)+V_{t,h}(dx,A),
\end{align}
where $U_{t,h}(.,A), V_{t,h}(.,A)$ are finite measures on $[-B,B]$.
However, an obvious estimate and Theorem \ref{lem:minimumLaw} give that
\begin{align}\label{eq:V}
\limsup\ttinf{t}V_{t,h}([-B,B],A)\leq||F||_{\infty}\lim\ttinf{t}\Qb^{(h)}\lbrb{\ind{m_t\leq-A}} =||F||_{\infty}\lbrb{1-h\int_{0}^{A}e^{-ha}da}.
\end{align}
Like in any of the previous proofs and especially \eqref{eq:minimum1} we have that in the sense of measures
\begin{align}\label{eq:UsimtildeU}
\nonumber &U_{t,h}(dx,A)=\int_{\omega\in C(0,\infty)}F(\omega)\ind{w_u\in dx}\ind{m_t(\omega)\geq -A}Q^{(h)}_t(d\omega)=\\
\nonumber &\frac{1}{\Eb_0(e^{-hC_t})}\int_{\omega\in C(0,\infty)}F(\omega)\ind{w_u\in dx}\ind{m_t(\omega)\geq -A}e^{-C_t(\omega)}W^h_t(d\omega)=\\
\nonumber & \frac{\Eb_0\lbbrbb{e^{X_{th^2}-C_{th^2}}F\lbrb{\frac{X_{.h^2}}{h}}1_{\{X_{uh^2}\in hdx\}}\ind{m_{th^2}\geq-Ah}}}{\Eb_0(e^{-hC_t})}\sim
\\ &th^2\Eb_0\lbbrbb{e^{X_{th^2}-C_{th^2}}F\lbrb{\frac{X_{.h^2}}{h}}1_{\{X_{uh^2}\in hdx\}}\ind{m_{th^2}\geq-Ah}}=:\tilde{U}(dx,A).  
\end{align}
Moreover to evaluate the latter we follow with immediate modifications \eqref{eq:minimum2} to get 
\begin{align}\label{eq:U1U2}
\nonumber&\tilde{U}(dx,A)=th^2\int_{0}^{Ah}e^{-a}\int_{a}^{\infty}e^{-c}\Eb_a\lbbrbb{e^{X_{th^2}}O(X)1_{\{X_{uh^2}\in hdx+a\}}\Tc_{(0,c)}>th^2}dcda+\\
\nonumber&th^2e^{-Ah}\int_{Ah}^{\infty}e^{-c}\Eb_{Ah}\lbbrbb{e^{X_{th^2}}O(X)1_{\{X_{uh^2}\in hdx+a\}}\Tc_{(0,c)}>th^2}dc=\\
&U^1_{t,h}(dx,A)+U^2_{t,h}(dx,A),
\end{align}
where for the sake of brevity we have put $O(X)=F\lbrb{\frac{X_{.h^2}-a}{h}}$. Clearly, we have from Lemma \ref{lem:Density} that
\begin{align} \label{eq:U2}
\nonumber&\limsup\ttinf{t}U^2_{t,h}(\lbbrbb{-B,B},A)\leq\\ &\limsup\ttinf{t}||F||_\infty e^{-Ah}\lbrb{th^2\int_{Ah}^{\infty}e^{-c}\Eb_{Ah}\lbbrbb{e^{X_{th^2}},\Tc_{(0,c)}>th^2}dc}=
||F||_\infty Ahe^{-Ah}.
\end{align}
Since \[\lim\ttinf{A}\limsup\ttinf{t}th^2\lbrb{U^2_{t,h}(\lbbrbb{-B,B},A)+V_{t,h}([-B,B],A)}=0\]
and \eqref{eq:UsimtildeU} holds it suffices to study $th^2U^1_{t,h}(dx,A)$.
Using the Markov property at time $uh^2$ above we get
\begin{align*}
&th^2U^1_{t,h}(dx,A)=\\
 &th^2\int_{0}^{Ah}e^{-a}\int_{a}^{\infty}e^{-c}\Eb_a\lbbrbb{O(X)1_{\{X_{uh^2}\in hdx+a\}}1_{\{\Tc_{(0,c)}>uh^2\}}}\Eb_{a+hx}\lbbrbb{e^{X_{th^2-uh^2}}1_{\{\Tc_{(0,c)}>th^2-uh^2\}}}dadc.
\end{align*}
Assuming the validity of Lemma \ref{lem:U} below and using the asymptotic relation \eqref{eq:UsimtildeU} we get that, for any Borel measurable $\Cc\subset[-B,B]$,
\begin{align*}
&  \lim\ttinf{t}\int_{x\in\Cc}th^2U^1_{t,h}(dx,A)=\lim\ttinf{t}\Eb^{\Qb^{(h)}_t}\lbbrbb{F(X_.)\ind{X_u\in \Cc}\ind{m_t\geq-A}}=\\
&h^2\int_{\Cc}\int_{0\vee (-x)}^{Ah}ae^{-ah}\Eb^{\dagger}_a\lbbrbb{F_u(X_.-a)1_{\{X_u\in dx+a\}}\}}da.
\end{align*}
Setting $A\uparrow\infty$ we then get
\begin{align*}
&  \lim\ttinf{t}\Eb^{\Qb^{(h)}_t}\lbbrbb{F(X_.)\ind{X_u\in \Cc}}=
h^2\int_{\Cc}\int_{0\vee (-x)}^{\infty}ae^{-ah}\Eb^{\dagger}_a\lbbrbb{F_u(X_.-a)1_{\{X_u\in dx+a\}}\}}da.
\end{align*}
This concludes the proof as $B>0$ is arbitrary and this holds for any bounded positive measurable functional $F$ and any $u>0$. However, the last expression corresponds to a shifted to zero three dimensional Bessel process $Y^{e_h}-e_h$, started from independent random variable with distribution $\Pbb{e_h\in dx}=h^2x e^{-hx}dx,x>0$. This concludes the proof of the theorem.
\end{proof}
To study the measure $th^2U^1_{t,h}(dx,A)$ we prove the following proposition.
\begin{lemma}\label{lem:U}
We have that for any $x\in[-B,B]$, $A>2B\lbrb{h+h^{-1}}$
\begin{equation}
\lim\ttinf{t}th^2U^1_{t,h}(dx,A)=\int_{0\vee (-x)}^{A}h^2ae^{-ha}\Eb^{\dagger}_a\lbbrbb{F_u(X_.-a)1_{\{X_u\in dx+a\}}\}}da
\end{equation}

\end{lemma}
\begin{proof}
We consider and estimate in the sense of measures
\begin{align*}
&th^2\tilde{U}^1_{t,h}(dx,A)=\\
&th^2\int_{0}^{Ah} e^{-a}\int_{a}^{\ln(t)}e^{-c}\Eb_a\lbbrbb{O(X)1_{\{X_{uh^2}\in hdx+a\}}1_{\{\Tc_{(0,c)}>uh^2\}}}\Eb_{a+hx}\lbbrbb{e^{X_{th^2-uh^2}}1_{\{\Tc_{(0,c)}>th^2-uh^2\}}}dadc\leq\\
&\int_{0}^{Ah} e^{-a}\Eb_a\lbbrbb{O(X)1_{\{X_{uh^2}\in hdx+a\}}1_{\{\Tc_{(0,\infty)}>uh^2\}}}\lbrb{th^2\int_{a}^{\ln(t)}e^{-c}\Eb_{a+hx}\lbbrbb{e^{X_{th^2-uh^2}}1_{\{\Tc_{(0,c)}>th^2-uh^2\}}}dc}da
\end{align*}
Then elementary modification of Lemma \ref{lem:improvedConvergence} and \eqref{eq:estimatefoUtilde} whenever $\ln(t)>Ah+Bh=\sup (a+hx)$ yields
\begin{align}\label{eq:Utilde}
&th^2\tilde{U}^1_{t,h}([-B,B],A)\leq ||F||_\infty o\lbrb{1}.
\end{align}
Therefore, it remains to study the remaining portion of the integral, or the measure
\begin{align}\label{eq:measureUhat}
&th^2\hat{U}^1_{t,h}\lbrb{dx,A} =\\
\nonumber&th^2\int_{0}^{Ah} e^{-a}\int_{\ln(t)}^{\infty}e^{-c}\Eb_a\lbbrbb{O(X)1_{\{X_{uh^2}\in hdx+a\}}1_{\{\Tc_{(0,c)}>uh^2\}}}\Eb_{a+hx}\lbbrbb{e^{X_{th^2-uh^2}}1_{\{\Tc_{(0,c)}>th^2-uh^2\}}}dcda.
\end{align}
Splitting on the event  $\{\Tc_{(0,\infty)}>uh^2\}$ we get that
\begin{align*}
&th^2\int_{0}^{Ah} e^{-a}\int_{\ln(t)}^{\infty}e^{-c}\Eb_a\lbbrbb{O(X)1_{\{X_{uh^2}\in hdx+a\}}1_{\{\Tc_{(0,c)}>uh^2\}}}\Eb_{a+hx}\lbbrbb{e^{X_{th^2-uh^2}}1_{\{\Tc_{(0,c)}>th^2-uh^2\}}}dcda=\\
&th^2\int_{0}^{Ah} e^{-a}\Eb_a\lbbrbb{O(X)1_{\{X_{uh^2}\in hdx+a\}}1_{\{\Tc_{(0,\infty)}>uh^2\}}}\int_{\ln(t)}^{\infty}e^{-c}\Eb_{a+hx}\lbbrbb{e^{X_{th^2-uh^2}}1_{\{\Tc_{(0,c)}>th^2-uh^2\}}}dcda-\\
&th^2\int_{0}^{Ah} e^{-a}\int_{\ln(t)}^{\infty}e^{-c}\Eb_a\lbbrbb{O(X)1_{\{X_{uh^2}\in dx+a\}}1_{\{\Tc_{(0,\infty)}>uh^2\cap\Tc_{(0,c)}\leq uh^2 \}}}\times\\
&\Eb_{a+hx}\lbbrbb{e^{X_{th^2-uh^2}}1_{\{\Tc_{(0,c)}>th^2-uh^2\}}}dcda=\\
&S(t,dx,A)-\tilde{S}(t,dx,A),
\end{align*}
where we note that $a\geq 0\vee(-hx)$ since otherwise we have that the impossible inequality $m_t> X_u,u\leq t$ must hold, namely the running minimum to exceed the value of the process.
However, according to Lemma \ref{lem:improvedConvergence} and the uniform convergence in \eqref{eq:lemDensity1} we get that, for $a+hx>0, a\in(0,A),x\in[-B,B]$, 
\begin{align}\label{eq:EstimateNow}
\lim\ttinf{t}\sup_{a+hx>0, a\in(0,A),x\in[-B,B]}\labsrabs{th^2\int_{\ln(t)}^{\infty}e^{-c}\Eb_{a+hx}\lbbrbb{e^{X_{th^2-uh^2}}1_{\{\Tc_{(0,c)}>th^2-uh^2\}}}dc-a-hx}=0
\end{align} 
and henceforth
\begin{align}\label{eq:uniformity}
\nonumber &\lim\ttinf{t}S(t,dx,A)=\lim\ttinf{t}\Bigg (\int_{0\vee (-hx)}^{Ah} e^{-a}\Eb_a\lbbrbb{O(X)1_{\{X_{uh^2}\in hdx+a\}}1_{\{\Tc_{(0,\infty)}>uh^2\}}}\times\\
\nonumber&\Big(th^2\int_{\ln(t)}^{\infty}e^{-c}
\Eb_{a+hx}\lbbrbb{e^{X_{th^2-uh^2}}1_{\{\Tc_{(0,c)}>th^2-uh^2\}}}dc\Big)da\Bigg)=\\
&\int_{0\vee (-hx)}^{Ah}e^{-a}(a+hx)\Eb_a\lbbrbb{O(X)1_{\{X_{uh^2}\in hdx+a\}}1_{\{\Tc_{(0,\infty)}>uh^2\}}}da
\end{align}
However \eqref{eq:semigroupBessel} allows us to deduct that
\[(a+hx)\Eb_a\lbbrbb{O(X)1_{\{X_{uh^2}\in hdx+a\}}1_{\{\Tc_{(0,\infty)}>uh^2\}}}=a\Eb^{\dagger}_a\lbbrbb{O(X)1_{\{X_{uh^2}\in hdx+a\}}}.\]
We show using \eqref{eq:scalingBessel}, $O(X)=F\lbrb{\frac{X_{.h^2}-a}{h}}$, the rescaling property for the Bessel process and lastly changing variables $\frac{a}{h}\mapsto a$ that we have
\begin{align*}
&\lim\ttinf{t}S(t,dx,A)=\int_{0\vee (-hx)}^{Ah}ae^{-a}\Eb^{\dagger}_a\lbbrbb{F\lbrb{\frac{X_{.h^2}-a}{h}}1_{\{X_{uh^2}\in hdx+a\}}\}}da=\\
&\int_{0\vee (-hx)}^{Ah}ae^{-a}\Eb^{\dagger}_{\frac{a}{h}}\lbbrbb{F\lbrb{X_.-\frac{a}{h}}1_{\{X_u\in dx+\frac{a}{h}\}}\}}da=h^2\int_{0\vee (-x)}^{A}ae^{-ha}\Eb^{\dagger}_{a}\lbbrbb{F\lbrb{X_.-a}1_{\{X_u\in dx+a\}}\}}da
\end{align*}
To conclude that 
\[\lim\ttinf{t}th^2\hat{U}^1_{t,h}\lbrb{dx,A}=\lim\ttinf{t}th^2U^1_{t,h}\lbrb{dx,A}=\lim\ttinf{t}S\lbrb{t,dx,A}\]
it remains to show that $\lim\ttinf{t}\tilde{S}\lbrb{t,dx,A}$ is the zero measure. First note that for any fixed $0\leq a\leq A$  and $c>\ln(t)$ we have that in sense of  measures
\begin{align*}
&\Eb_a\lbbrbb{O(X)1_{\{X_{uh^2}\in hdx+a\}}1_{\{\Tc_{(0,\ln(t))}>uh^2\}}}\leq \Eb_a\lbbrbb{O(X)1_{\{X_{uh^2}\in hdx+a\}}1_{\{\Tc_{(0,c)}>uh^2\}}}\leq\\
&\Eb_a\lbbrbb{O(X)1_{\{X_{uh^2}\in hdx+a\}}1_{\{\Tc_{(0,\infty)}>uh^2\}}}.
\end{align*}

 This for the first inequality together with  the estimate \eqref{eq:EstimateNow} for the second give
\begin{align*}
&\tilde{S}(t,[-B,B],A)\leq \\
&\int_{x=-B}^{B}th^2\Big(\int_{0\vee (-hx)}^{A} e^{-a}\Eb_a\lbbrbb{O(X)1_{\{X_{uh^2}\in hdx+a\}}1_{\{\Tc_{(0,\infty)}>uh^2\cap\Tc_{(0,\ln(t))}\leq uh^2 \}}}\times\\
&\int_{\ln(t)}^{\infty}e^{-c}\Eb_{a+hx}\lbbrbb{e^{X_{th^2-uh^2}}1_{\{\Tc_{(0,c)}>th^2-uh^2\}}}dcda \Big)\leq\\
&\int_{x=-B}^{B}\Big(\int_{0\vee (-hx)}^{A} e^{-a}\Eb_a\lbbrbb{O(X)1_{\{X_{uh^2}\in hdx+a\}}1_{\{\Tc_{(0,\infty)}>uh^2\cap\Tc_{(0,\ln(t))}\leq uh^2 \}}}\times\\
&\sup_{a+hx>0, a\in(0,A),x\in[-B,B]}(a+hx+o(1))da \Big)\leq\\
|&2B(Bh+A+o(1))||F||_\infty\sup_{0<a<A}\Pb_a\lbrb{\Tc_{(0,\infty)}>uh^2\cap\Tc_{(0,\ln(t))}\leq uh^2 \}}\leq\\
&2B(Bh+A+o(1))||F||_\infty\sup_{0<a<A}\Pb_a\lbrb{\Tc_{\{\ln(t)\}}\leq uh^2}\leq\\
&2B(Bh+A+o(1))||F||_\infty\Pb_A\lbrb{\Tc_{\{\ln(t)\}}\leq uh^2}=o(1),
\end{align*}
where $\Tc_{\{\ln(t)\}}=\inf\lbcurlyrbcurly{s>0:\,X_s=\ln(t)}$.
\end{proof}

\section{Poisson summation and the function $G(v,x)$}\label{sec:PS}
We consider the Fourier transform defined as follows
\begin{equation}\label{eq:Fourier}
\hat {f}(\xi):=\int_{-\infty}^{\infty}e^{-2\pi \xi i x}f(x)dx.
\end{equation}
We recall that if $\labsrabs{f(x)}+\labsrabs{\hat f(x)}\leq C(1+|x|)^{-1-\delta}$ for some $\delta>0,C>0$ and $\forall x\in\Rb$ then
\begin{equation}\label{eq:PoissonSummation}
\suminfinf f(j+x)=\suminfinf \hat f(j)e^{2i\pi j x}.
\end{equation}
When $f(x)=\frac{1}{\sqrt{2\pi\sigma^2}}e^{-\frac{x^2}{2\sigma^2}},\,x\in\Rb$ then $\hat f(\xi)=e^{-2\pi^2 \xi^2\sigma^2}, \xi\in\Rb$ and the function clearly admits Poisson summation thanks to its rapid decay at infinity.

Define for $v>0, x\in \lbbrbb{0,1}$
\begin{align}\label{eq:F}
&G(v,x)=2\sumoneinf \cos(2\pi j x)\expeigenvalue{v}=\suminfinf\cos(2\pi j x)\expeigenvalue{v}-1.
\end{align} 
Then the following result is a standard consequence of the Poisson summation.
\begin{lemma}\label{lem:G(v,x)}
For any $v>0$, 
\begin{align}\label{eq:FPoisson}
&G(v,x)=\frac{\sqrt{2}}{\sqrt{\pi v}} \suminfinf e^{-2\frac{\lbrb{j-x}^2}{v}} -1.
\end{align} 
If $x\in\lbrb{0,1}$ then, as $v\to 0$, for any $l\geq 0, l\in\Nb\cup\{0\}$
\begin{equation}\label{eq:GsmallAsymp}
\frac{\partial^l G}{\partial v^l}(v,x)=G^{(l)}(v,x)\sim\frac{2^l\sqrt{2}}{\sqrt{\pi}v^{2l+\frac{1}{2}}}\suminfinf\lbrb{j-x}^{2l}e^{-2\frac{(j-x)^2}{v}}.
\end{equation}
If $x\in\{0,1\}$ then, as $v\to 0$, for any $l\geq 0, l\in\Nb\cup\{0\}$
\begin{equation}\label{eq:GsmallAsymp1}
\frac{\partial^l G}{\partial v^l}(v,x)=G^{(l)}(v,x)\sim\minusone{l}\frac{\sqrt{2}l!}{\sqrt{\pi}v^{l+\frac{1}{2}}}.
\end{equation}
\end{lemma}
\begin{proof}
To justify \eqref{eq:FPoisson} we apply the Poisson summation for $f(x)=\frac{1}{\sqrt{2\pi\sigma^2}}e^{-x^2/2\sigma^2}$ with $\sigma=\sqrt{v}/2$ to get
\begin{align*}
&G(v,x)=2\sumoneinf \cos\lbrb{2\pi j x}\expeigenvalue{v}=\suminfinf\cos\lbrb{2\pi j x}\expeigenvalue{v}-1\\
&=\suminfinf \cos\lbrb{2\pi j x}\expeigenvalue{v}-1=\frac{\sqrt{2}}{\sqrt{\pi v}} \suminfinf e^{-2\frac{\lbrb{j-x}^2}{v}} -1.
\end{align*}
The relations \eqref{eq:GsmallAsymp} and \eqref{eq:GsmallAsymp1} are a result of differentiation of \eqref{eq:FPoisson} which applies due to the uniform convergence of \eqref{eq:F} in any small enough neighbourhood of $v>0$.
\end{proof}

\section{The function $F(v,t)$}\label{sec:FF}

We recall that the Mellin transform is defined as follows
\begin{equation}\label{eq:Mellin}
\Mcc f(s):=\IntZeroInf x^{s-1}f(x)dx.
\end{equation}
Then Mellin transform is well defined at least for all $s$ such that $\Mcc |f|(\Re(s))<\infty$. If for example $\Mcc f(s)$ is defined,
absolutely integrable and uniformly decaying to zero along the lines of the strip $a<c:=\Re(s)<b$, for
$a < b$, the Mellin inversion theorem applies as follows
\begin{equation}\label{eq:MellinInversion}
f(x):=\twopi \IntComplexLine{\Mcc f(s)x^{-s}ds}{c},
\end{equation}
for any $a<c<b$. We recall that with $f_a(x)=\lbrb{1+x}^{-a}$, for any $a>0$, we have that
\begin{equation}\label{eq:MellinPower}
\Mcc f_a(s)=\frac{\Gamma(s)\Gamma(a-s)}{\Gamma(a)}, \text{for all $s:0<\Re(s)<a$.}
\end{equation}
We note the special case that will be needed further which follow from the 
\begin{equation}\label{eq:MellinPower2}
\Mcc f_2(s)=\frac{\pi\lbrb{1-s}}{\sin\lbrb{\pi s}}, \text{ for $0<c<2$.}
\end{equation}
We know that, as $\theta\to\infty$, the following asymptotic holds
\begin{equation}\label{eq:expDecaySin}
\frac{1}{|\sin\lbrb{\pi(c+i\theta)}|}\sim Ce^{-\pi\labsrabs{\theta}}.
\end{equation}
Therefore  $\Mcc f_2$ is invertible on its region of definition.\\
Recall that from \eqref{eq:FF} we have that by definition
\[F(v,t)=\sumoneinf\minusone{j+1}\frac{\pij}{\lbrb{\pij \frac{v}{t}+1}^2}\expeigenvalue{v}.\]
We see that using formally \eqref{eq:MellinPower2} and \eqref{eq:MellinInversion} with $x=\pij \frac{v}{t}$, for any $0<c=\Re(s)<1$, we obtain that
\begin{align*}
&F(v,t)=\twopi\sumoneinf\minusone{j+1}\pij\IntComplexLine{v^{-s}t^s\pi^{-2s}j^{-2s}\frac{\pi(1-s)}{\sin(\pi s)}ds}{c}\expeigenvalue{v}=\\
&\twopi\IntComplexLine{v^{-s}t^{s} \frac{\pi(1-s)}{\sin(\pi s)}\sumoneinf\minusone{j+1}\pi^{2-2s}j^{2-2s}\expeigenvalue{v}ds}{c}.
\end{align*}
The interchange of integration and summation is justified by the fact that 
\begin{align*}
&\int_{-\infty}^{\infty}\sumoneinf j^{2-2c}\expeigenvalue{v}\frac{\labsrabs{|\theta|+1}}{\labsrabs{\sin(\pi(c+i\theta))}}d\theta<\infty,
\end{align*}
which in turn follows from \eqref{eq:expDecaySin}. We denote by 
\begin{equation}\label{eq:eta}
\eta(s,v)=\sumoneinf\minusone{j+1}\pi^{2-2s}j^{2-2s}\expeigenvalue{v}
\end{equation}
and note that $\eta(s,v)$ is clearly an entire function for any $v>0$. Then we have that
\begin{equation}\label{eq:Frepresentation}
F(v,t)=
\twopi\IntComplexLine{v^{-s}t^{s} \frac{\pi(1-s)}{\sin(\pi s)}\eta(s,v)ds}{c}.
\end{equation}
Clearly, from the reflection formula the poles of $\pi/\sin(\pi s)$, for $\Re(s)<1$, are located at $0,-1,-2,-3,\cdots$ and it has residues at each pole of value $\minusone{n}$. Since \eqref{eq:expDecaySin} holds we can use the residue theorem to conclude that upon shifting the contour from $c\in(0,1)$ to $c\in(-n-1,-n)$ that
\begin{equation}\label{eq:Frepresentation1}
F(v,t)=\eta(0,v)+\sum_{l=1}^{n}\minusone{l}\lbrb{l+1}\frac{v^{l}}{t^{l}}\eta(-l,v)+\twopi\IntComplexLine{v^{-s}t^{s} \frac{\pi(1-s)}{\sin(\pi s)}\eta(s,v)ds}{c}.
\end{equation}

Next we investigate the properties of $\eta(s,v)$ where we recall that $s=c+i\theta$. We check from \eqref{eq:eta} immediately that, for $c\in(-n-1,-n)$ with $\{c\}=-c-n$, the following representation of $\eta$ is available
\begin{align}\label{eq:etaIntegralRep}
\nonumber&\eta(c+i\theta,v)=\sumoneinf\minusone{j+1}(\pij)^{n+2}(\pij)^{\{c\}-1-i\theta}\expeigenvalue{v}=\\
\nonumber&2^{c+1-i\theta}\sumoneinf\minusone{j+1}(\pij)^{n+2}\frac{1}{\Gamma(\{c\}+1-i\theta)}\int_{0}^{\infty}\expeigenvalue{u}u^{\{c\}-i\theta}du\expeigenvalue{v}=\\
\nonumber& \frac{2^{c+1-i\theta}}{\Gamma(\{c\}+1-i\theta)} \int_{0}^{\infty}\lbrb{\sumoneinf\minusone{j+1}(\pij)^{n+2}\expeigenvalue{(u+v)}}u^{\{c\}-i\theta}du=\\
&\frac{2^{c+1-i\theta}}{\Gamma(\{c\}+1-i\theta)} \int_{0}^{\infty}\eta(-1-n,u+v)u^{\{c\}-i\theta}du.
\end{align}
We proceed to study in more detail $\eta(-l,v)$.
\begin{lemma}\label{lem:eta}
We have that $\eta(-l+1,v)=\minusone{l-1}2^{l-1}G^{(l)}\lbrb{v,\frac{1}{2}}, l\geq 1$ and $\eta(0,v)=F(v,0)$. As $v\to 0$,
 \begin{equation}\label{eq:etaAsympSmall}
 \eta(-l+1,v)=\minusone{l-1}2^{l-1}G^{(l)}\lbrb{v,\frac{1}{2}}\sim \minusone{l-1}\frac{2^{2l-1}\sqrt{2}}{\sqrt{\pi}v^{2l+\frac{1}{2}}}\suminfinf\lbrb{j-\frac{1}{2}}^{2l}e^{-2\frac{\lbrb{j-\frac{1}{2}}^2}{v}}
 \end{equation}
 and, as $v\to\infty$,
  \begin{equation}\label{eq:etaAsympLarge}
     \eta(-l+1,v)=\minusone{l-1}2^{l-1}G^{(l)}\lbrb{v,\frac{1}{2}}\sim \minusone{l}\pi^{2l}e^{-\pi^2 v}.
  \end{equation}
\end{lemma}
\begin{proof}
The representation of $\eta(-l+1,v)$ follows by formal differentiation in \eqref{eq:F} with $x=1/2$ and inspection of the terms. Finally the proof of \eqref{eq:etaAsympLarge} follows from differentiation of \eqref{eq:F} and \eqref{eq:etaAsympSmall} is a result of differentiation and \eqref{eq:GsmallAsymp}.
\end{proof}
We are now ready to obtain our crucial result.
\begin{lemma}\label{lem:FAsympExpansion}
We have that, for any $n\in\Nb$, and uniformly for compact sets of $v$ the following asymptotic expansions hold
\begin{equation}\label{eq:FAsympExpansion}
F(v,t)\sim \eta(0,v)+\sum_{l=1}^{n}\minusone{l}\lbrb{l+1}\frac{v^{l}}{t^{l}}\eta(-l,v)+h(v)o\lbrb{\frac{1}{t^n}}
\end{equation}
and 
\begin{equation}\label{eq:IntFAsympExpansion}
\IntZeroInf F(v,t)dv\sim 1+\sum_{l=1}^{n}\minusone{l}\frac{2^{1}\lbrb{l+1}!}{t^{l}}+o\lbrb{\frac{1}{t^n}}
\end{equation}
\end{lemma}
\begin{proof}
Recall that $s=c+i\theta$. Relation \eqref{eq:FAsympExpansion} holds immediately from \eqref{eq:Frepresentation1} and the fact that from \eqref{eq:eta} 
\[\labsrabs{\eta(s,v)}\leq \sumoneinf \pi^{2c}j^{2c}\expeigenvalue{v}<\infty. \]
To prove \eqref{eq:IntFAsympExpansion} we observe that all terms involving $v^{l}\eta(-l+1,v)$ are absolutely integrable thanks to \eqref{eq:etaAsympSmall} and \eqref{eq:etaAsympLarge}. Indeed at infinity all is clear from \eqref{eq:etaAsympLarge} whereas we apply \eqref{eq:etaAsympSmall} as follows ignoring any constants with respect to $v$: 
\begin{align*}
&\int_{0}^{1}v^{l}\labsrabs{\eta(-l+1,v)}dv\lesssim\int_{0}^{1}\frac{1}{v^{l+1}}\suminfinf\lbrb{j-\frac{1}{2}}^{2l}e^{-2\frac{\lbrb{j-\frac{1}{2}}^2}{v}}dv=\\
&\suminfinf\int_{1}^{\infty}\lbrb{j-\frac{1}{2}}^{2l} v^{l-1}e^{-2\lbrb{j-\frac{1}{2}}^2v}dv=\suminfinf \int_{\lbrb{j-\frac{1}{2}}^{2}}^{\infty} v^{l-1}e^{-2v}dv<\infty.
\end{align*}
 Also we conclude from $\eta(0,v)=G'\lbrb{v,\frac{1}{2}}$ that
\[\IntZeroInf \eta(0,v)dv=G\lbrb{\infty,\frac{1}{2}}-G\lbrb{0,\frac{1}{2}}=1,\]
where the latter integration to $1$ can be concluded from \eqref{eq:F} and \eqref{eq:FPoisson}. The other terms in \eqref{eq:IntFAsympExpansion}, namely
\[\IntZeroInf v^{l}\eta\lbrb{-l,v}dv=2^l l!\]
 can be deduced by using that $\eta(-l,0)=\minusone{l}2^{1}G^{(l+1)}\lbrb{v,\frac{1}{2}}$ from Lemma \ref{lem:eta}, integration by parts which holds due to \eqref{eq:etaAsympSmall} and \eqref{eq:etaAsympLarge}. 

 So it remains to consider the integral term. Invoking \eqref{eq:etaIntegralRep} we note that
\begin{align*}
&\tilde H(v):=\labsrabs{\twopi\IntComplexLine{v^{-s}t^{s} \frac{\pi(1-s)}{\sin(\pi s)}\eta(s,v)ds}{c}} \leq\\
&\frac{2^c t^c}{v^c} \int_{-\infty}^{\infty} \frac{|1-c+i\theta|}{\labsrabs{\sin(\pi (c+i\theta))\Gamma\lbrb{\{c\}+1-i\theta}}}\int_{0}^{\infty}\labsrabs{\eta(-n-1,u+v)}u^{\{c\}}dud\theta. 
\end{align*} 
Upon integration with respect to $v$ and changing variables $x=u+v,y=u$ we get
\begin{align*}
&\IntZeroInf \tilde H(v)dv\leq\\
& 2^ct^c \int_{-\infty}^{\infty} \frac{|1-c+i\theta|}{\labsrabs{\sin(\pi (c+i\theta))\Gamma\lbrb{\{c\}+1-i\theta}}}\IntZeroInf \labsrabs{\eta(-n-1,x)}\int_{0}^{x}y^{\{c\}}(x-y)^{-c}dydxd\theta\leq\\
&2^ct^c Beta(\{c\}+1,1-c)\int_{-\infty}^{\infty} \frac{|1-c+i\theta|}{\labsrabs{\sin(\pi (c+i\theta))\Gamma\lbrb{\{c\}+1-i\theta}}}\IntZeroInf \labsrabs{\eta(-n-1,x)}x^{1+n}dxd\theta<\infty,
\end{align*}
where we have used that $c=-\{c\}-n$.
The finiteness of the last integral is a similar consequence from \eqref{eq:etaAsympSmall} and \eqref{eq:etaAsympLarge} as before and the decay of $\labsrabs{\sin(\pi (c+i\theta))}\sim Ce^{\pi\labsrabs{\theta}}$, see \eqref{eq:expDecaySin} which surpasses that of $\labsrabs{\Gamma\lbrb{\{c\}+1-i\theta}}\asymp e^{-\frac{\pi}{2}\labsrabs{\theta}}$.
\end{proof}
\section{The function $J_1(a,t,\nu)$}\label{sec:J_1}
\subsection{The auxillary functions $\Gamma$ and $H$}\label{subsec:GammaH}

We recall that 
\begin{equation}\label{eq:HRecall}
H\lbrb{a,u,\rho,\gamma,h,1}=-\sumoneinf \frac{\pi j}{\pi^2j^2u\rho+1}\expeigenvalue{u\gamma}\sin\lbrb{\pi j+\pi j a h},
\end{equation}
see \eqref{eq:H}. Denote by
\begin{align}\label{eq:underlyingFunction}
\nonumber&\Gamma\lbrb{a,u,\rho,\gamma,h,1}=\sumoneinf \frac{1}{\pij u\rho+1}\expeigenvalue{u\gamma}\cos\lbrb{\pi j a h+\pi j}\\
\nonumber&\sumoneinf \IntZeroInf\expeigenvalue{u\gamma}e^{-\pij u\rho v-v} \cos\lbrb{2\pi j\lbrb{\frac{ah}{2}+\frac{1}{2}}}dv=\\
\nonumber&\IntZeroInf\sumoneinf \expeigenvalue{u\gamma}e^{-\pij u\rho v} \cos\lbrb{2\pi j\lbrb{\frac{ah}{2}+\frac{1}{2}}} e^{-v}dv=\\
&\frac{1}{2}\IntZeroInf G\lbrb{u\gamma+2u\rho v,\frac{ah+1}{2}}e^{-v}dv,
\end{align}
where the interchange of integration and summation is obviously possible for any fixed pair $u>0,\gamma>0$ and we have used \eqref{eq:F} to identify the expressions with $G(\cdot,\cdot)$. Applying \eqref{eq:FPoisson} we see further that \[\Gamma\lbrb{a,u,\rho,\gamma,h,1}=\frac{1}{2}\IntZeroInf\lbrb{\frac{1}{\sqrt{2\pi u\lbrb{\gamma+2\rho v}}}} \suminfinf e^{-\frac{2\lbrb{j-\frac{ah}{2}-\frac{1}{2}}^2}{u\gamma+2\frac{vu}{t}}} e^{-v}dv-\frac{1}{2}\]
and thus 
\begin{align}\label{eq:derivativeGamma1}
\nonumber &\frac{\partial \Gamma}{\partial \gamma}\lbrb{a,u,\frac{1}{t},1,\frac{\sqrt{u}}{\sqrt{t}},1}=\\
\nonumber-&\frac{1}{4\sqrt{2\pi u}}\IntZeroInf \frac{1}{\lbrb{1+2\frac{v}{t}}^{\frac{3}{2}}}\suminfinf e^{-\frac{2\lbrb{j-\frac{a\frac{\sqrt{u}}{\sqrt{t}}}{2}-\frac{1}{2}}^2}{u+2\frac{vu}{t}}}e^{-v}dv +\\
&\frac{1}{2\sqrt{2\pi u^3}}\IntZeroInf \frac{1}{\lbrb{1+2\frac{v}{t}}^{\frac{5}{2}}}\suminfinf\lbrb{j-\frac{a\frac{\sqrt{u}}{\sqrt{t}}}{2}-\frac{1}{2}}^2 e^{-\frac{2\lbrb{j-\frac{a\frac{\sqrt{u}}{\sqrt{t}}}{2}-\frac{1}{2}}^2}{u+2\frac{vu}{t}}}e^{-v}dv,
\end{align}
where the interchange of the derivative in $\gamma$, for any $u>0$, and the integral is clear due to the absolute integrability of the expressions under the integrals above. This expression \eqref{eq:derivativeGamma1} will be useful when $u\leq 1$. Otherwise, when $u\geq 1$, we use the following which is immediate from the definition of $\Gamma$, namely
 \begin{align}\label{eq:derivativeGamma2}
&\frac{\partial \Gamma}{\partial \gamma}\lbrb{a,u,\frac{1}{t},1,\frac{\sqrt{u}}{\sqrt{t}},1}=-\frac{u}{2}\sumoneinf \frac{\pij}{\pij \frac{u}{t}+1}\expeigenvalue{u}\cos\lbrb{\pi j a\frac{\sqrt{u}}{\sqrt{t}}+\pi j}.
\end{align}
  Clearly upon differentiation we get
\begin{equation}\label{eq:HandGamma}
\frac{\partial H}{\partial a}\lbrb{a,u,\frac{1}{t},1,\frac{\sqrt{u}}{\sqrt{t}},1}=\frac{2}{t^{1/2}u^{1/2}}\frac{\partial \Gamma}{\partial \gamma}\lbrb{a,u,\frac{1}{t},1,\frac{\sqrt{u}}{\sqrt{t}},1}.
\end{equation}
These representations allow for the following claim
\begin{proposition}\label{prop:unifromIntegrability}
We have that, for any $t>100a^2$,
\begin{equation}\label{eq:unifromIntegrability}
\IntZeroInf \frac{1}{\sqrt{u}}\labsrabs{\frac{\partial H}{\partial a}\lbrb{a,u,\frac{1}{t},1,\frac{\sqrt{u}}{\sqrt{t}},1}}du=\frac{2}{\sqrt{t}}\IntZeroInf\frac{1}{u}\labsrabs{\frac{\partial \Gamma}{\partial \gamma}\lbrb{a,u,\frac{1}{t},1,\frac{\sqrt{u}}{\sqrt{t}},1}}du <\infty
\end{equation}
and even 
\begin{equation}\label{eq:unifromBounds}
\sup_{t>100a^2}\sup_{0\leq b\leq a}\frac{1}{\sqrt{u}}\labsrabs{\frac{\partial H}{\partial a}\lbrb{b,u,\frac{1}{t},1,\frac{\sqrt{u}}{\sqrt{t}},1}}\leq f(u)
\end{equation}
\begin{equation}\label{eq:unifromBounds1}
\sup_{t>100a^2}\sup_{0\leq b\leq a}\frac{2}{u}\labsrabs{\frac{\partial \Gamma}{\partial \gamma}\lbrb{b,u,\frac{1}{t},1,\frac{\sqrt{u}}{\sqrt{t}},1}}\leq f(u)
\end{equation}
with $\IntZeroInf f(u)du<\infty.$
\end{proposition}
\begin{proof}
We start with \eqref{eq:unifromBounds1}. Clearly, when $u\geq 1$, a trivial bound using \eqref{eq:derivativeGamma2} gives  \eqref{eq:unifromBounds1} with 
\[f(u):=u\suminfinf \pij\expeigenvalue{u} \]
 and  $\int_{1}^{\infty}f(u)du<\infty$. Assume that $u\leq 1$. Then $\sup_{0 \leq b\leq a}b\frac{\sqrt{u}}{\sqrt{t}}\leq 1/10$ for $t>100a^2,u\leq 1$, and  using this in \eqref{eq:derivativeGamma1} the following estimate is obtained
\begin{align*}
\sup_{t\geq 100a^2}\sup_{0 \leq b\leq a}\labsrabs{\frac{2}{u}\frac{\partial \Gamma}{\partial \gamma}\lbrb{b,u,\frac{1}{t},1,\frac{\sqrt{u}}{\sqrt{t}},1}}\leq \frac{C}{u^{\frac{5}{2}}}\IntZeroInf\suminfinf \lbrb{j-\frac{1}{3}}^2e^{-2\frac{\lbrb{j-\frac{1}{3}}^2}{u\lbrb{1+\frac{v}{50a^2}}}} e^{-v}dv:=f(u)
\end{align*}
for some $C>0$ big enough. Put $b_j=b_j(v,a):=\frac{\lbrb{j-\frac{1}{3}}^2}{\lbrb{1+\frac{v}{50a^2}}}$ and we get upon changing variables $u\mapsto 1/w$ that
\begin{align*}
&\int_{0}^{1}f(u)du=C\suminfinf \lbrb{j-\frac{1}{3}}^2 \IntZeroInf \int_{1}^{\infty}\sqrt{w}e^{-wb_j}dwe^{-v}dv=\\
&C\suminfinf \lbrb{j-\frac{1}{3}}^2 \IntZeroInf b^{-\frac{3}{2}}_{j} \int_{b_j}^{\infty}\sqrt{w}e^{-w}dwe^{-v}dv\leq \\
&C\suminfinf \frac{1}{\labsrabs{j-\frac{1}{3}}}\int_{0}^{\infty}\lbrb{1+\frac{v}{50a^2}}^{\frac{3}{2}}\int_{b_j}^{\infty}\sqrt{w}e^{-w}dwe^{-v}dv.
\end{align*}
Splitting the integration in $v$ at the point where $b_j=\labsrabs{j-\frac{1}{3}}$, namely $v_j=50a^2\labsrabs{j-\frac{1}{3}}-1$, we get that
\begin{align*}
&\int_{0}^{1}f(u)du\leq C\IntZeroInf \lbrb{1+\frac{v}{50a^2}}^{\frac{3}{2}}e^{-v}dv\times \suminfinf \frac{1}{\labsrabs{j-\frac{1}{3}}}\int_{\labsrabs{j-\frac{1}{3}}}^{\infty}\sqrt{w}e^{-w}dw+\\
& C\IntZeroInf \sqrt{w}e^{-w}dw\times\suminfinf \frac{j^2}{\lbrb{j-\frac{1}{3}}^3}\int_{v_j}^{\infty}\lbrb{1+\frac{v}{50a^2}}^{\frac{3}{2}}e^{-v}dv<\infty.
\end{align*}
This proves \eqref{eq:unifromBounds1}.\\
When $u\leq 1$, $t>100a^2$, we get from \eqref{eq:HandGamma} and \eqref{eq:derivativeGamma1} that
\begin{align*}
&\labsrabs{\frac{\partial H}{\partial a}\lbrb{a,u,\frac{1}{t},1,\frac{\sqrt{u}}{\sqrt{t}},1}}\leq \frac{C}{u^{\frac{3}{2}}}\IntZeroInf \suminfinf \lbrb{j-\frac{1}{3}}^2e^{-2\frac{\lbrb{j-\frac{1}{3}}^2}{u\lbrb{1+\frac{v}{50a^2}}}}e^{-v}dv
\end{align*}
and the rest is the same as in \eqref{eq:unifromBounds1}. Therefore
\[\IntZeroOne\frac{1}{\sqrt{u}}\labsrabs{\frac{\partial H}{\partial a}\lbrb{a,u,\frac{1}{t},1,\frac{\sqrt{u}}{\sqrt{t}},1}}du<\infty.\]
When $u>1$ we use \eqref{eq:HandGamma} and \eqref{eq:derivativeGamma2} and the fact that $t>100a^2$ to get easily that
\begin{align*}
&\labsrabs{\frac{\partial H}{\partial a}\lbrb{a,u,\frac{1}{t},1,\frac{\sqrt{u}}{\sqrt{t}},1}}\leq Cue^{-u}
\end{align*}
and henceforth
\[\int_{1}^\infty\frac{1}{\sqrt{u}}\labsrabs{\frac{\partial H}{\partial a}\lbrb{a,u,\frac{1}{t},1,\frac{\sqrt{u}}{\sqrt{t}},1}}du<\infty.\]
Inequality \eqref{eq:unifromBounds} is immediate from the computations above and \eqref{eq:HandGamma}. Finally, \eqref{eq:unifromIntegrability} follows from  \eqref{eq:unifromBounds} and \eqref{eq:unifromBounds1}. 
\end{proof}
\subsection{Main results on $J_1(a,t,\nu)$}\label{subsec:J_1Results}
For any $0<\nu\leq \infty$ define the function
\begin{align}\label{eq:J_1General}
\nonumber &J_1(t,a,\nu):=\frac{1}{\sqrt{t}}\int_{\frac{t}{\lbrb{a+\nu\sqrt{t}}^2}}^{\frac{t}{a^2}} \frac{1}{\sqrt{u}}\sumoneinf \minusone{j+1}\frac{\pi j}{\pi^2j^2\frac{u}{t}+1}\expeigenvalue{u} \eigenfunction{\pi j a \frac{\sqrt{u}}{\sqrt{t}}}du=\\
&\frac{1}{\sqrt{t}}\int_{\frac{t}{\lbrb{a+\nu\sqrt{t}}^2}}^{\frac{t}{a^2}}\frac{1}{\sqrt{u}}H\lbrb{a,u,\frac{1}{t},1,\frac{\sqrt{u}}{\sqrt{t}},1}du.
\end{align}
 Note that $J_1(t,a,\infty)=J_1(t,a)$ in \eqref{eq:J_1}.
We are now ready to study $J_1(t,a,\nu)$. 
\begin{lemma}\label{lem:J_1}
For any $0<\nu\leq \infty$, we have that
\begin{equation}\label{eq:J_1limit}
\lim\ttinf{t}tJ_1(t,a,\nu)=\lim\ttinf{t}\int_{0}^{a}t\frac{\partial J_1(t,b,\nu)}{\partial b}db=-aG\lbrb{\frac{1}{\nu^2},\frac{1}{2}},
\end{equation}
where $G\lbrb{\nu,\frac{1}{2}}$ is defined in \eqref{eq:F}. Moreover, for some    function $f\colon[0,\infty)\mapsto[0,\infty)$
\begin{equation}\label{eq:uniformBoundsToProve1}
\sup_{t>100a^2}\sup_{b\leq a}\labsrabs{tJ_1(t,b,\nu)}\leq a\IntZeroInf f(u)du<\infty
\end{equation}
and the convergence in \eqref{eq:J_1limit} is uniform on $a$-compact intervals.
\end{lemma}
\begin{proof}

We will discuss $t\partial J_1(t,a,\nu)/\partial a$ showing that it converges to $1$, as $t\to\infty$. Then since $J_1(t,0,\nu)=0$ we get the answer by using the DCT in
\[J_1(t,a,\nu)=\int_{0}^{a}\frac{\partial J_1(t,b,\nu)}{\partial b}db.\]
First note that thanks to the definition of $J_1(t,a,\nu)$, Proposition \ref{prop:unifromIntegrability} and \eqref{eq:HandGamma} we obtain that
\begin{align}\label{eq:derivativeJ_1}
\nonumber &t\frac{\partial J_1(t,a,\nu)}{\partial a}=2\int_{\frac{t}{\lbrb{a+\nu\sqrt{t}}^2}}^{\frac{t}{a^2}} \frac{1}{u}\frac{\partial \Gamma}{\partial \gamma}\lbrb{a,u,\frac{1}{t},1,\frac{\sqrt{u}}{\sqrt{t}},1}du-\\
&\frac{2t}{a^2}H\lbrb{a,\frac{t}{a^2},\frac{1}{t},1,\frac{1}{a},1}+\frac{2t}{\lbrb{a+\nu\sqrt{t}}^2}H\lbrb{a,\frac{t}{\lbrb{a+\nu\sqrt{t}}^2},\frac{1}{t},1,\frac{1}{a+\nu\sqrt{t}},1}.
\end{align}
However, from \eqref{eq:HRecall} we get that
\begin{align*}
&\frac{2t}{a^2}H\lbrb{a,\frac{t}{a^2},\frac{1}{t},1,\frac{1}{a},1}=\frac{2t}{a^2}\sumoneinf \frac{\pi j}{\pi^2j^2a^{-2}+1}\expeigenvalue{\frac{t}{a^2}}\sin\lbrb{2\pi j}=0.
\end{align*}
When $\nu=\infty$ we clearly have \[\frac{2t}{\lbrb{a+\nu\sqrt{t}}^2}H\lbrb{a,\frac{t}{\lbrb{a+\nu\sqrt{t}}^2},\frac{1}{t},1,\frac{1}{a+\nu\sqrt{t}},1}=0.\]
Let us consider $0<\nu<\infty$. Then from \eqref{eq:HRecall} we get the bound
\begin{align*}\label{eq:Hbound}
\nonumber&\sup_{0\leq b\leq a}\frac{2t}{\lbrb{b+\nu\sqrt{t}}^2}\labsrabs{H\lbrb{b,\frac{t}{\lbrb{b+\nu\sqrt{t}}^2},\frac{1}{t},1,\frac{1}{b+\nu\sqrt{t}},1}}\leq\\
\nonumber& \frac{1}{\nu^2}\sup_{0\leq b\leq a}\sumoneinf \frac{\pi j}{\frac{\pij}{\lbrb{b+\nu\sqrt{t}}^2}+1}e^{-\frac{\pij}{2\lbrb{b+\nu\sqrt{t}}^2}t}\labsrabs{\sin\lbrb{\pi j+\pi j\frac{b}{b+\nu\sqrt{t}}}}\leq\\
\nonumber&\frac{1}{\nu^2}\sumoneinf \pi je^{-\frac{\pij}{4\lbrb{\frac{a^2}{t}+\nu^2}}}\sup_{0\leq b\leq a}\labsrabs{\sin\lbrb{\pi j\frac{b}{b+\nu\sqrt{t}}}}\leq\\
&\frac{a}{\nu^3\sqrt{t}}\sumoneinf \pij e^{-\frac{\pij}{4\lbrb{\frac{a^2}{t}+\nu^2}}}=o(1),
\end{align*}
which means that
\begin{equation}\label{eq:J1error}
\sup_{0\leq b\leq a}\labsrabs{t\frac{\partial J_1(t,b,\nu)}{\partial b}-2\int_{\frac{t}{\lbrb{b+\nu\sqrt{t}}^2}}^{\frac{t}{b^2}} \frac{1}{u}\frac{\partial \Gamma}{\partial \gamma}\lbrb{b,u,\frac{1}{t},1,\frac{\sqrt{u}}{\sqrt{t}},1}du}\leq\frac{a}{\nu^3\sqrt{t}}\sumoneinf \pij e^{-\frac{\pij}{4\lbrb{\frac{a^2}{t}+\nu^2}}}=o\lbrb{1}. 
\end{equation}
Henceforth  \eqref{eq:unifromBounds1} allows us to apply the DCT to demonstrate that
\[\lim\ttinf{t}t\frac{\partial J_1(t,a,\nu)}{\partial a}=2\int_{\frac{1}{\nu^2}}^{\infty} \frac{1}{u}\frac{\partial \Gamma}{\partial \gamma}\lbrb{a,u,0,1,0,1}du.\]
However, from \eqref{eq:derivativeGamma2},
\[\frac{2}{u}\frac{\partial \Gamma}{\partial \gamma}\lbrb{a,u,0,1,0,1}=\frac{2}{u}\frac{\partial \Gamma}{\partial \gamma}\lbrb{0,u,0,1,0,1}=G'\lbrb{u,\frac{1}{2}}=\sumoneinf \minusone{j+1}\pij \expeigenvalue{u}. \]
Finally from \eqref{eq:FF} we recognize that the last sum is simply $F(u,0)$ which according to \eqref{eq:F} leads to $F(u,0)=G'\lbrb{u,\frac{1}{2}}$. Therefore since $G\lbrb{\infty,\frac{1}{2}}=0$, see \eqref{eq:F},
\[\lim\ttinf{t}t\frac{\partial J_1(t,a,\nu)}{\partial a}=\int_{\frac{1}{\nu^2}}^{\infty}F(u,0)du=G\lbrb{\infty,\frac{1}{2}}-G\lbrb{\frac{1}{\nu^2},\frac{1}{2}}=-G\lbrb{\frac{1}{\nu^2},\frac{1}{2}}.\] 
Moreover, \eqref{eq:unifromIntegrability} ensures that even 
\begin{align*}
&\sup_{b\leq a} \labsrabs{t\frac{\partial J_1(t,b,\nu)}{\partial b}}\leq\IntZeroInf\frac{2}{u} \labsrabs{\frac{\partial \Gamma}{\partial \gamma}\lbrb{b,u,\frac{1}{t},1,\frac{\sqrt{u}}{\sqrt{t}},1}}du\leq \IntZeroInf f(u)du<\infty
\end{align*}
and therefore the dominated convergence theorem applies and yields our claim namely \eqref{eq:J_1limit}.
Even more this uniform bound on the derivative gives \eqref{eq:uniformBoundsToProve1} and subsequently the uniform convergence in \eqref{eq:J_1limit} for $a$-compact sets.
\end{proof}
When $\nu<\infty$ we are able to give some other useful estimates.
\begin{proposition}\label{prop:J1UniversalBound}
Let $\infty>\nu>0$. Then, for any $h>0$, we have that
\begin{equation}\label{eq:J1UniversalBound}
\labsrabs{J_1(t,a,\nu}\leq \frac{1}{a}\sumoneinf\pij e^{-\frac{\pij}{2\lbrb{\frac{a}{\sqrt{t}}+\nu}^2}}=\frac{2}{a}G'\lbrb{\frac{1}{\lbrb{\frac{a}{\sqrt{t}}+\nu}^2},0},
\end{equation}
where $G$ is defined in \eqref{eq:F}.
\end{proposition}
\begin{proof}
The proof is immediate from \eqref{eq:J_1General} and $|\sin(x)|\leq |x|$. Indeed note that using this we get
\begin{align*}
&\labsrabs{J_1(t,a,\nu}\leq \frac{a}{\sqrt{t}}\int_{\frac{t}{\lbrb{a+\nu\sqrt{t}}^2}}^{\frac{t}{a^2}} \frac{1}{\sqrt{u}}\sumoneinf \pi^2j^2\expeigenvalue{u}  \frac{\sqrt{u}}{\sqrt{t}}du\leq\\
&\frac{a}{t} \frac{t}{a^2}\sumoneinf \pi^2j^2e^{-\frac{\pij}{2\lbrb{\frac{a}{\sqrt{t}}+\nu}^2}}\leq \frac{2}{a}\sumoneinf \pi^2j^2e^{-\frac{\pij}{2\lbrb{\frac{a}{\sqrt{t}}+\nu}^2}},
\end{align*}
which proves the assertion.
\end{proof}
The next result allows us to improve the uniform convergence proved in Lemma \ref{lem:J_1}.
\begin{proposition}\label{prop:J_1Improvement}
Let $\infty>\nu>0$ and $a(t)\uparrow \infty$ such that $a(t)=o\lbrb{t^{\frac{1}{4}}}$ then for some $C>0$
\begin{equation}\label{eq:J1errorImprovement}
R(t):=\sup_{0\leq b\leq a(t)}\labsrabs{t\frac{\partial J_1(t,b,\nu)}{\partial b}-2\int_{\frac{1}{\nu^2}}^{\infty} \frac{1}{u}\frac{\partial \Gamma}{\partial \gamma}\lbrb{0,u,0,1,0,1}du}\leq C\lbrb{\frac{a^2(t)}{t}+\frac{a(t)}{\sqrt{t}}}.
\end{equation}
and 
\begin{equation}\label{eq:J1errorImprovement1}
r(t):=\sup_{0\leq b\leq a(t)}\labsrabs{t J_1(t,b,\nu)-2b\int_{\frac{1}{\nu^2}}^{\infty} \frac{1}{u}\frac{\partial \Gamma}{\partial \gamma}\lbrb{0,u,0,1,0,1}du}\leq C\lbrb{\frac{a^3(t)}{t}+\frac{a^2(t)}{\sqrt{t}}}
\end{equation}
\end{proposition}
\begin{proof}
From \eqref{eq:J1error} which is valid, for any $a>0$, we see that the following bound can be immediately derived
\begin{align}\label{eq:R1}
\nonumber &R(t)\leq \frac{Ca(t)}{\sqrt{t}}+\\
\nonumber&\sup_{0\leq b\leq a(t)}\labsrabs{2\int_{\frac{1}{\nu^2}}^{\infty} \frac{1}{u}\frac{\partial \Gamma}{\partial \gamma}\lbrb{0,u,0,1,0,1}du-2\int_{\frac{t}{\lbrb{b+\nu\sqrt{t}}^2}}^{\frac{t}{b^2}} \frac{1}{u}\frac{\partial \Gamma}{\partial \gamma}\lbrb{b,u,\frac{1}{t},1,\frac{\sqrt{u}}{\sqrt{t}},1}du}=\\
&\frac{Ca(t)}{\sqrt{t}}+R_1(t).
\end{align}
Next we study $R_1(t)$. Consider
\[\hat R_{1}(t)=\sup_{0\leq b\leq a(t)}\labsrabs{\int_{\frac{t}{\lbrb{b+\nu\sqrt{t}}^2}}^{\frac{t}{b^2}}\lbrb{\frac{2}{u}\frac{\partial \Gamma}{\partial \gamma}\lbrb{0,u,0,1,0,1}du- \frac{2}{u}\frac{\partial \Gamma}{\partial \gamma}\lbrb{b,u,\frac{1}{t},1,\frac{\sqrt{u}}{\sqrt{t}},1}}du}.\]
Then from \eqref{eq:derivativeGamma2} we easily get writing 
\begin{align*}
&- \frac{2}{u}\frac{\partial \Gamma}{\partial \gamma}\lbrb{b,u,\frac{1}{t},1,\frac{\sqrt{u}}{\sqrt{t}},1}=\sumoneinf \frac{\pij}{\pij \frac{u}{t}+1}\expeigenvalue{u}\cos\lbrb{\pi j b\frac{\sqrt{u}}{\sqrt{t}}+\pi j}=\\
&\sumoneinf \pij\expeigenvalue{u}\cos\lbrb{\pi j b\frac{\sqrt{u}}{\sqrt{t}}+\pi j}- \frac{u}{t}\sumoneinf \frac{\pi^4j^4}{\pij \frac{u}{t}+1}\expeigenvalue{u}\cos\lbrb{\pi^4j^4 b\frac{\sqrt{u}}{\sqrt{t}}}
\end{align*}
 that the following inequalities hold
\begin{align*}
&\hat R_{1}(t)=\sup_{0\leq b\leq a(t)}\labsrabs{\int_{\frac{t}{\lbrb{b+\nu\sqrt{t}}^2}}^{\frac{t}{b^2}}\lbrb{\frac{2}{u}\frac{\partial \Gamma}{\partial \gamma}\lbrb{0,u,0,1,0,1}du- \frac{2}{u}\frac{\partial \Gamma}{\partial \gamma}\lbrb{b,u,\frac{1}{t},1,\frac{\sqrt{u}}{\sqrt{t}},1}}du}\leq\\
 \sup_{0\leq b\leq a(t)}&\int_{\frac{t}{\lbrb{b+\nu\sqrt{t}}^2}}^{\frac{t}{b^2}}\labsrabs{\sumoneinf \frac{\pij}{\pij \frac{u}{t}+1}\expeigenvalue{u}\cos\lbrb{\pi j\lbrb{ b\frac{\sqrt{u}}{\sqrt{t}}+1}}-\sumoneinf\pij\expeigenvalue{u}\cos\lbrb{\pi j}}du\\
&\leq \frac{1}{t}\sup_{0\leq b\leq a(t)}\int_{\frac{t}{\lbrb{b+\nu\sqrt{t}}^2}}^{\frac{t}{b^2}}u\sumoneinf \pi^4j^4\expeigenvalue{u}du+\\
\sup_{0\leq b\leq a(t)}&\int_{\frac{t}{\lbrb{b+\nu\sqrt{t}}^2}}^{\frac{t}{b^2}}\labsrabs{\sumoneinf\pij\expeigenvalue{u}\lbrb{\cos\lbrb{\pi j\lbrb{ b\frac{\sqrt{u}}{\sqrt{t}}+1}}-\cos\lbrb{\pi j}}}du\leq \\
&\frac{1}{t}\sumoneinf\int_{\frac{t}{\lbrb{a(t)+\nu\sqrt{t}}^2}}^{\infty}u \pi^4j^4\expeigenvalue{u}du+\frac{a^2(t)}{t}\sumoneinf\int_{\frac{t}{\lbrb{a(t)+\nu\sqrt{t}}^2}}^{\infty}u \pi^4j^4\expeigenvalue{u}du\leq\\
&\frac{4a^2(t)}{t}\sumoneinf\int_{\frac{1}{2\nu^2}}^{\infty}u \pi^4j^4\expeigenvalue{u}du\leq C\frac{a^2(t)}{t},
\end{align*}
where we have used implicitly $|\cos(x+\pi j)-\cos(\pi j)|=|1-\cos(x)|\leq x^2/2$. Next we estimate using \eqref{eq:derivativeGamma2} that
\begin{align*}
&\tilde R_{1}(t)=\sup_{0\leq b\leq a(t)}\labsrabs{\int_{\frac{t}{\lbrb{b+\nu\sqrt{t}}^2}}^{\frac{1}{\nu^2}}\frac{2}{u}\frac{\partial \Gamma}{\partial \gamma}\lbrb{0,u,0,1,0,1}du}\leq \int_{\frac{1}{\lbrb{\frac{a(t)}{\sqrt{t}}+\nu}^2}}^{\frac{1}{\nu^2}}\labsrabs{\frac{2}{u}\frac{\partial \Gamma}{\partial \gamma}\lbrb{0,u,0,1,0,1}}du\leq \\
&\lbrb{\frac{1}{\nu^2}-\frac{1}{\lbrb{\frac{a(t)}{\sqrt{t}}+\nu}^2}}\suminfinf \pij \expeigenvalue{\frac{1}{\lbrb{\frac{a(t)}{\sqrt{t}}+\nu}^2}}\leq \frac{C}{\nu^4}\lbrb{\frac{a^2(t)}{t}+\frac{a(t)}{\sqrt{t}}}.
\end{align*}
Finally, using again \eqref{eq:derivativeGamma2} we obtain
\begin{align*}
& \bar R_{1}(t)=\sup_{0\leq b\leq a(t)}\labsrabs{\int_{\frac{t}{b^2}}^{\infty}\frac{2}{u}\frac{\partial \Gamma}{\partial \gamma}\lbrb{0,u,0,1,0,1}du}\leq\\
&\sumoneinf\int_{\frac{t}{a^2(t)}}^{\infty}\pij \expeigenvalue{u} du=
\sumoneinf e^{-\pij \frac{t}{2a^2(t)}}\leq C\frac{a^2(t)}{t}.
\end{align*}
Since $R_1(t)\leq \tilde{R}_1(t)+\bar{R}_1(t)+\hat R_1(t)$ we get that
\[R_1(t)\leq C\lbrb{\frac{a^2(t)}{t}+\frac{a(t)}{\sqrt{t}}}\]
and \eqref{eq:J1errorImprovement} follows. Now \eqref{eq:J1errorImprovement1} is an immediate consequence of \eqref{eq:J1errorImprovement} applied in the sequence of inequalities
\begin{align*}
r(t)=&\sup_{0\leq b\leq a(t)}\labsrabs{t J_1(t,b,\nu)-2b\int_{\frac{1}{\nu^2}}^{\infty} \frac{1}{u}\frac{\partial \Gamma}{\partial \gamma}\lbrb{0,u,0,1,0,1}du}\leq\\
&\sup_{0\leq b\leq a(t)}\int_{0}^{b}\labsrabs{\frac{\partial J_1(t,c,\nu)}{\partial c}-2\int_{\frac{1}{\nu^2}}^{\infty} \frac{1}{u}\frac{\partial \Gamma}{\partial \gamma}\lbrb{0,u,0,1,0,1}du}dc\leq\\
&\int_{0}^{a(t)}\sup_{0\leq b\leq a(t)}\labsrabs{\frac{\partial J_1(t,c,\nu)}{\partial c}-2\int_{\frac{1}{\nu^2}}^{\infty} \frac{1}{u}\frac{\partial \Gamma}{\partial \gamma}\lbrb{0,u,0,1,0,1}du}dc\leq a(t)R(t).
\end{align*}
\end{proof}
\section{The function $J_2(a,t)$}\label{sec:J_2}
Define, for any $0<\nu\leq \infty$,
\begin{align}\label{eq:J_2Recall}
\nonumber&J_2(t,a,\nu):=\frac{1}{\sqrt{t}}\int_{\frac{t}{\lbrb{a+\nu\sqrt{t}}^2}}^{\frac{t}{a^2}} \frac{1}{\sqrt{u}}\sumoneinf \frac{\pi j}{\pi^2j^2\frac{u}{t}+1}\expeigenvalue{u} \eigenfunction{\pi j a \frac{\sqrt{u}}{\sqrt{t}}}e^{-\frac{\sqrt{t}}{\sqrt{u}}}du=\\
&\frac{1}{\sqrt{t}}\int_{\frac{t}{\lbrb{a+\nu\sqrt{t}}^2}}^{\frac{t}{a^2}}\frac{1}{\sqrt{u}}H\lbrb{a,u,\frac{1}{t},1,\frac{\sqrt{u}}{\sqrt{t}},0}du,
\end{align}
where $H\lbrb{a,u,\frac{1}{t},1,\frac{\sqrt{u}}{\sqrt{t}},0}$ is defined in \eqref{eq:H}.
\begin{lemma}\label{lem:J2}
We have that
\begin{equation}\label{eq:J2Bound}
\sup_{\nu>0}\sup_{0\leq b< \infty}t\labsrabs{J_2(t,b,\nu)}=o\lbrb{1}.
\end{equation}
\end{lemma}
\begin{proof}

 From \eqref{eq:J_2Recall} we easily get the estimate 
\begin{equation}\label{eq:estimateJ_2}
\sup_{\nu>0}\sup_{0\leq b<\infty}t\labsrabs{J_2(t,b,\nu)}\leq \sqrt{t}\IntZeroInf \frac{1}{\sqrt{u}}\sumoneinf \pij \expeigenvalue{u}e^{-\frac{\sqrt{t}}{\sqrt{u}}}du=-\sqrt{t}\IntZeroInf \frac{1}{\sqrt{u}}G'\lbrb{u,0}e^{-\frac{\sqrt{t}}{\sqrt{u}}}du.
\end{equation}
We then split this integral into three regions.

\textbf{Region $u\leq 1$:}
We know from $\eqref{eq:GsmallAsymp1}$ that $\labsrabs{G'(u,0)}\sim Cu^{-\frac{3}{2}}$. Then this feeds in \eqref{eq:estimateJ_2} to yield
\begin{align}\label{eq:J_2estimate1}
\nonumber&\sqrt{t}\IntZeroOne\frac{1}{\sqrt{u}}\sumoneinf \pij \expeigenvalue{u}e^{-\frac{\sqrt{t}}{\sqrt{u}}}du=\sqrt{t} \IntZeroOne\frac{1}{\sqrt{u}}\labsrabs{G'(u,0)}e^{-\frac{\sqrt{t}}{\sqrt{u}}}du\leq\\
& \sqrt{t}e^{-\frac{\sqrt{t}}{2}}\IntZeroOne\frac{1}{\sqrt{u}}\labsrabs{G'(u,0)}e^{-\frac{1}{2\sqrt{u}}}du=o(1)
\end{align}
\textbf{Region $1<u\leq t^{\frac{1}{4}}$:} For this region we directly estimate
\begin{align}\label{eq:J_2estimate2}
&\sqrt{t}\int_{1}^{t^{\frac{1}{4}}}\frac{1}{\sqrt{u}}\sumoneinf \pij \expeigenvalue{u}e^{-\frac{\sqrt{t}}{\sqrt{u}}}du\leq \sqrt{t}e^{-t^{\frac{1}{4}}}\sumoneinf \pij \expeigenvalue{}\int_{1}^{t^{\frac{1}{4}}}e^{-\lbrb{u-1}}du=o(1).
\end{align}
\textbf{Region $t^{\frac{1}{4}}<u<\infty$:} This part is also easily estimated as follows
\begin{align}\label{eq:J_2estimate3}
&\sqrt{t}\int_{t^{\frac{1}{4}}}^{\infty}\frac{1}{\sqrt{u}}\sumoneinf \pij \expeigenvalue{u}e^{-\frac{\sqrt{t}}{\sqrt{u}}}du\leq \sqrt{t}e^{-\frac{t^{\frac{1}{4}}}{4}}\int_{t^{\frac{1}{4}}}^{\infty}\lbrb{\sumoneinf \pij \expeigenvalue{\frac{u}{4}}}du=o(1).
\end{align}
Collecting \eqref{eq:J_2estimate1}, \eqref{eq:J_2estimate2} and \eqref{eq:J_2estimate3} and plugging them in \eqref{eq:estimateJ_2} we prove \eqref{eq:J2Bound}.
\end{proof}
\section*{Acknowledgement}
We would like to thank Titus Hilberdink for helping us to simplify the proof of Lemma 1.1 considerably. We would also thank Prof. Velenik for pointing out the references \cite{IV12} and \cite{KM12} and explaining the available results in the area.

\end{document}